\documentclass[11pt]{amsart}
\usepackage{amsmath,amscd,amssymb,amsfonts,amsthm,hyperref}
\usepackage[all]{xy}

 \newtheorem{theorem}{Theorem}[section]

\newtheorem{proposition}[theorem]{Proposition}

\newtheorem{lemma}[theorem]{Lemma}
\theoremstyle{remark}
\newtheorem{remark}[theorem]{Remark}
\newtheorem{definition}[theorem]{Definition}

\newtheorem{example}[theorem]{Example}
\newtheorem{examples}[theorem]{Examples}
\renewcommand\AA{\mathbb{A}}

\newcommand\PP{\mathbb{P}}

\renewcommand\L{\mathcal{L}}

\renewcommand{\TH}{\mathbb{T}H}

\newcommand{\T}{\mathbb{T}}

\newcommand{\n}{\mathfrak{n}}

\newcommand{\pr}{\on{pr}}

\newcommand\lie[1]{\mathfrak{#1}}
\renewcommand{\k}{\lie{k}}
\newcommand{\h}{\lie{h}}

\newcommand{\g}{\lie{g}}

\newcommand{\q}{\lie{q}}
\newcommand{\on}{\operatorname}

\newcommand{\Ad}{ \on{Ad} }
\newcommand{\ad}{ \on{ad} }

\newcommand{\End}{ \on{End} }

\renewcommand{\ker}{ \on{ker}}

\newcommand{\Mult}{{\on{Mult}}}
\newcommand{\id}{{\on{id}}}

\newcommand{\da}{\dasharrow}

\newcommand{\qu}{/\!\!/}

\newcommand{\hra}{\hookrightarrow}
\newcommand{\xra}{\xrightarrow}

\newcommand{\rra}{\rightrightarrows}

\newcommand{\dd}{\mf{d}}
\newcommand{\nn}{\mf{n}}
\newcommand{\uu}{\mf{u}}
\newcommand{\pp}{\mf{p}}
\newcommand{\rr}{\mf{r}}

\newcommand{\ol}{\overline}

\newcommand\sig{\sigma}
\newcommand\eps{\epsilon}
\newcommand\Om{\Omega}

\newcommand{\f}{\tfrac}
\newcommand{\cc}{\mf{c}}

\renewcommand{\l}{\langle}
\renewcommand{\r}{\rangle}
\newcommand{\hh}{{ \f{1}{2}}}
\newcommand{\ti}{\tilde}

\newcommand\pt{\on{pt}}

\newcommand\beqn{\begin{equation}}
\newcommand\eeqn{\end{equation}}
\newcommand{\ca}{\mathcal}
\newcommand{\wh}{\widehat}
\newcommand{\wt}{\widetilde}
\newcommand{\mf}{\mathfrak}
\newcommand{\beq}{\begin{eqnarray*}}
\newcommand{\eeq}{\end{eqnarray*}}

\newcommand{\Cour}[1]      {[\![#1]\!]}

\newcommand{\VB}{{\mathcal{VB}}}
\newcommand{\LA}{{\mathcal{LA}}}
\newcommand{\CA}{{\mathcal{CA}}}
\newcommand{\sz}{\mathsf{s}}
\newcommand{\tz}{\mathsf{t}}
\newcommand{\uz}{\mathsf{u}}
\renewcommand{\subset}{\subseteq}

\begin{document}

\title[]{Dirac actions and Lu's Lie algebroid}

\author{E. Meinrenken}\address{University of Toronto, Department of Mathematics, 40 St George Street, Toronto, Ontario M4S2E4, Canada }\email{mein@math.toronto.edu}

\date{\today}

\begin{abstract}
Poisson actions of Poisson Lie groups have an interesting and rich geometric structure. We will generalize some of this structure to Dirac actions of Dirac Lie groups. Among other things, we extend a result of Jiang-Hua-Lu, which states that the cotangent Lie algebroid and the action algebroid for a Poisson action form a matched pair. We also give a full classification of Dirac actions for which the base manifold is a homogeneous space $H/K$, obtaining a generalization of Drinfeld's classification for the Poisson Lie group case. 
\end{abstract}

\maketitle
\setcounter{tocdepth}{2}

{\small \tableofcontents \pagestyle{headings}}

\setcounter{section}{-1}
\vskip.3in
\section{Introduction}\label{sec:intro}
A Poisson-Lie group $H$ is a Lie group endowed with a Poisson 
structure $\pi_H$ such that the multiplication map $\on{Mult}_H\colon H\times H\to H$ is Poisson. An action of $H$ on a Poisson manifold $(M,\pi_M)$ is 
called a Poisson action if the action map $\ca{A}_M\colon H\times M\to M$ is a Poisson map. 
Such actions may be seen as `hidden' symmetries, not necessarily preserving the Poisson structure on $M$. They were first studied by Semenov-Tian-Shansky \cite{se:dr} in the context of soliton theory. Poisson actions have a rich geometric structure, developed in the work of many authors, including \cite{fer:int,fl:pc,kar:cla,lu:mo,lu:po,lu:note,xu:poigr}.

Let $L=T^*_\pi M$ be the cotangent Lie algebroid of $M$. Since the $H$-action on $M$ does not preserve 
the Poisson structure, its cotangent lift is not by Lie algebroid automorphisms, in general. However, 
by a result of Jiang-Hua Lu \cite{lu:po} there is a new Lie algebroid structure on the direct sum 
with the action Lie algebroid $M\times \h$
\[ \wh{L}=(M\times \h)\oplus T^*_\pi M\]
in such a way that the summands are Lie subalgebroids. This comes equipped with an $H$-action by Lie algebroid automorphisms preserving the first summand, and inducing the given action on the quotient
$\wh{L}/(M\times \h)\cong L$. Furthermore, letting $(\dd,\g,\h)_\beta$ be the Manin triple for the 
Poisson Lie group $(H,\pi_H)$ (with $\beta\in S^2\dd$ indicating the non-degenerate metric on $\dd^*\cong \dd$), the projection $M\times \h\to \h$ and the (symplectic) moment map 
$T^*_\pi M\to \h^*\cong\g$ combine into an $H$-equivariant Lie algebroid morphism 
\[ f_{\wh{L}}\colon \wh{L}\to \dd.\]
Among the applications developed in  \cite{lu:po} is simple and transparent discussion of Poisson homogeneous spaces. See \cite{bur:di1,bur:qua,kos:poi} for further aspects of Lu's Lie algebroid, and generalizations to quasi-Poisson actions. 

In this paper, we will consider Dirac actions of Dirac Lie groups. A \emph{Dirac manifold} $(M,\AA,E)$ is given by a manifold $M$ with a 
Dirac structure, that is, a Courant algebroid $\AA\to M$ together with an involutive Lagrangian subbundle $E\subseteq  \AA$. 
This implies in particular that $E$ is a Lie algebroid. Morphisms of Dirac manifolds are defined in terms of Lagrangian relations. A \emph{Dirac Lie group} is a Lie group $H$ equipped with a 
Dirac structure $(\AA,E)$, together with a 
multiplication morphism and a unit morphism 
\[ \Mult_\AA\colon (H,\AA,E)\times (H,\AA,E)\da (H,\AA,E),\ \ \ \eps_\AA\colon (\pt,0,0)\da (H,\AA,E)\]
satisfying the usual properties such as associativity. \emph{Dirac actions} of Dirac Lie groups on Dirac manifolds are defined similarly, in terms of  an action morphism
\[ \mathsf{a}_\PP\colon (H,\AA,E)\times (M,\PP,L)\da (M,\PP,L)
\] 
with base map an $H$-action on $M$. These definitions appear simpler than the approach in \cite{lib:dir} using $\VB$-groupoids, but as we will see (Theorem \ref{th:ingroupoid}) 
they are in fact equivalent:
\vskip.1in 
\noindent{\bf Theorem A.} {\it  For any Dirac Lie group $(H,\AA,E)$, the total space of $\AA$ has the structure of a $\VB$-groupoid 
$\AA\rra \AA^{(0)}$ over the group $H\rra \pt$, 
in such a way that $\Mult_\AA$ becomes the groupoid multiplication. Its space of units is 
$\AA^{(0)}=E_e$, the fiber at the group unit of $H$, and $E$ is a $\VB$-subgroupoid with the same space of units. Given a 
Dirac action on $(M,\PP,L)$, one obtains a $\VB$-groupoid action on $\PP$, with $\mathsf{a}_\PP$ as the action map; the subgroupoid $E\subset \AA$ preserves $L\subset \PP$.}
\vskip.1in

Suppose $\mathsf{a}_\PP$ is a given Dirac action of $(H,\AA,E)$ on $(M,\PP,L)$. 
Similar to the case of Poisson actions, the action of $H$ on $M$ lifts to an action on $(\PP,L)$, 
but this lift is usually not an action by Dirac automorphisms. We have the following generalization of Lu's result  (Theorem \ref{th:lu}): 
%In fact one does obtain an action $\bullet$ of $H$ on $\PP$, preserving $L$, but this action does not preserve the Courant algebroid bracket on $\PP$ not the Lie algebroid structure of $L$. 
\vskip.1in 
\noindent{\bf Theorem B.} {\it There is a new Dirac structure $(\wh{\PP},\wh{L})$ over $M$, on which $H$ acts by Dirac automorphisms, such that   
\[ \wh{\PP}=(M\times (\h\oplus\h^*))\oplus \PP,\ \ \ \wh{L}=(M\times \h)\oplus L\]
as vector bundles. The subbundle $(M\times \h^*)\oplus \PP$ is coisotropic and involutive, 
and reduction of $(\wh{\PP},\wh{L})$ with respect to this subbundle recovers 
$(\PP,L)$.}
\vskip.1in 
There is an analogous result, Theorem \ref{th:lulie}, for the category $\LA^\vee$ of Lie algebroids, with morphisms the $\LA$-comorphisms. (Dirac structures give examples by forgetting the ambient Courant algebroid; Poisson structures give examples by taking the cotangent Lie algebroid.) Given an $\LA^\vee$-action of $(H,E)$ 
on $(M,L)$, we show that the Lie algebroid structure on $\wh{L}=(M\times \h)\oplus L$ may be obtained as a quotient 
\[ \wh{L}=(TH\times L)/H\]
for a suitable $H$-action on $TH\times L$. The latter action is not by $\LA$-automorphisms, but it preserves the Lie bracket on \emph{$H$-invariant} sections.
As an application, we find that the space $E^{(0)}=E_e=:\g$ acquires a Lie algebra structure, the groupoid action of $E$ on its space of units is an $\LA^\vee$-action, 
and one obtains an $H$-equivariant Lie algebra triple $(\dd,\g,\h)$ with $\dd=\wh{\g}$.  
In \cite{lib:dir}, it is shown that the Dirac Lie group structures on a Lie group $H$ are classified by
such triples together with an $\ad$-invariant element $\beta\in S^2\dd$ such that $\g$ is $\beta$-coisotropic. 
A conceptual construction of $\beta$ along the lines of Theorem B is given in Section \ref{subsec:aq}. 
We refer to $(\dd,\g,\h)_\beta$ as an $H$-equivariant \emph{Dirac-Manin triple}. Returning to the case of a general Dirac action, we have: 
\vskip.1in
\noindent{\bf Theorem C.}\emph{
There is an $H$-equivariant bundle map $f_{\wh{\PP}}\colon \wh{\PP}\to \dd$, 
compatible with brackets, and with the additional property 
\[  f_{\wh{\PP}}(\gamma_{\wh{\PP}})=\beta,\]
where $\gamma_{\wh{\PP}}\in \Gamma(S^2\wh{\PP})$ is the element given by the 
metric. 
}
\vskip.1in
Here, compatibility with brackets means (cf.~ Section \ref{subsec:compatible}) that the map $f_{\wh{\PP}}$ together with the anchor map 
defines a bracket-preserving map $(\mathsf{a}_{\wh{\PP}},f_{\wh{\PP}})\colon \wh{\PP}\to TM\times \dd$, where the right hand side 
is regarded as a product Lie algebroid.

Similar to Lu's paper \cite{lu:poi}, we can use these results to classify the Dirac actions for the case $M=H/K$. Let $(H,\AA,E)$ be a Dirac Lie group, with corresponding Dirac-Manin triple $(\dd,\g,\h)_\beta$.  Let $\k$ be the Lie algebra of $K$. 
\vskip.1in
\noindent{\bf Theorem D.}\emph{
Dirac actions of $H$ on Dirac manifolds  $(M,\PP,L)$ with $M=H/K$ are classified by 
$K$-equivariant Manin pairs $(\nn,\uu)_{\gamma_\nn}$, with generators $\k\subseteq \uu$, 
together with $K$-equivariant Lie algebra morphisms $f_\nn\colon \nn\to \dd$ extending  the inclusion 
$\k\to \h$ and satisfying $f_\nn(\gamma_\nn)=\beta$.
}
\vskip.1in

%Given $(\PP,L)$, one obtains $\nn$ as the reduction of the coisotropic subspace $\ker(\mathsf{a}_{\wh{\PP}})|_{eK}$, with $\uu$ the reduction of $\wh{L}|_{eK}$ and with $f_\nn$ induced by $f_{\wh{\PP}}$. Conversely, 
Given these data, one recovers $(\PP,L)$ as a reduction of 
$(H\times \nn,\ H\times \uu)$ (an `action Dirac structure') by the action of $K$.  

 The classification takes on a simple form in the following special case. By general facts about $\VB$-groupoid actions (see Appendix \ref{app:vb}), the action of $E$ on $L$ dualizes to a $\VB$-groupoid action of $E^*$ on $L^*$. 
 Since $E$ is `vacant', $E^*$ is a Lie group. It turns out that $E^*$ is actually a Poisson Lie group, and the action is a Poisson action.  It is natural to assume that this action is transitive. We have the following generalization of Drinfeld's classification result, obtained in the thesis \cite{rob:cla}:
\vskip.1in 
 
\noindent{\bf Theorem.} (Robinson \cite{rob:cla}) {\it 
The Dirac structures $(\PP,L)$ on $M=H/K$, together with Dirac actions of $(H,\AA,E)$
such that  $E^*$ acts transitively on $L^*$, are classified by $K$-invariant $\beta$-coisotropic
Lie subalgebras $\cc\subseteq  \dd$ such that $\cc\cap \h=\k$.}
\vskip.1in
Some of our results are new even in the case of Poisson Lie groups, by considering Dirac actions of Poisson Lie groups on Dirac manifolds. An important class of 
Dirac Lie groups to which our results apply, and which is not a Poisson Lie group, 
is provided by the \emph{Cartan-Dirac structure} on any Lie group $H$ with an invariant metric on its Lie algebra $\h$. The Cartan-Dirac structure, discovered independently by Alekseev, Severa, and Strobl in the 1990s, is responsible for the theory of quasi-Hamiltonian spaces \cite{al:mom}, since the moment map condition can be described as a Dirac morphism to $H$ \cite{bur:di}. Its description as a Dirac Lie group (without using that terminology) is given in \cite[Theorem 3.9]{al:pur}; a future project will develop quasi-Hamiltonian spaces for more general Dirac Lie group targets. Quasi-Poisson Lie groups in the sense of \v{S}evera-Valach \cite{sev:lie} may also be considered from the perspective of Dirac Lie groups; see Section \ref{subsubsec:sv}.
Accordingly, Theorem D can serve as a starting point for the classification of the corresponding quasi-Poisson homogeneous spaces. 

We stress that our notion of `Dirac Lie groups' follows \cite{lib:dir}, and is similar to 
\cite{mil:two}, but is different from notions used in \cite{jot:dir,ort:mu}. See 
Section \ref{subsec:jotz}  below. Accordingly, Theorem D overlaps with Jotz'  \cite{jot:dir} classification results for `Dirac homogeneous spaces' only in the Poisson case.
\vskip.2in

\noindent{\bf Acknowledgements.} I would like to thank David Li-Bland for helpful discussions, as well as Patrick Robinson, whose thesis work inspired the techniques developed here. I also thank Madeleine Jotz for pointing out her work \cite{jot:digr} on Dirac Lie groupoids, and the referees for detailed comments.

\section{Dirac geometry}
In this section we review some background material  that will be needed in this paper. Our conventions will follow those of \cite{lib:cou,lib:dir,rob:cla}, to which we refer for a more detailed discussion. Throughout, a non-degenerate symmetric bilinear form on a vector space $V$ will be referred to as a \emph{metric}, and $V$ as a \emph{metrized vector space}. A \emph{metrized Lie algebra} (also known as a  \emph{quadratic Lie algebra}) 
is a Lie algebra with an $\ad$-invariant metric. 

\subsection{Dirac-Manin triples}\label{subsec:dmp}
We begin by describing some of the Lie-algebraic data that enter the classification results for Dirac Lie groups. See \cite[Section 3.2]{lib:dir} for further details. 
\subsubsection{}\label{subsubsec:coiso}
 Let $V$ be a vector space together with an element $\beta\in S^2V$. 
We denote by $\beta^\sharp\colon V^*\to V$ the map $\beta^\sharp(\mu)= \beta(\mu,\cdot)$. A subspace $U\subseteq V$ is called \emph{$\beta$-coisotropic} if $\beta$, viewed as a bilinear form on $V^*$, vanishes on the annihilator $\on{ann}(U)$. That is, 
$\beta^\sharp(\on{ann}(U))\subseteq U$. Equivalently, $\pr_{V/U}(\beta)=0$, where $\pr_{V/U}\colon V\to V/U$ is the quotient map. Note that the diagonal $V_\Delta\subseteq V\times V$ is $(\beta,-\beta)$-coisotropic. Given a linear map $\Phi\colon V\to V'$ with $\phi(\beta)=\beta'$, and 
a subspace $U'\subseteq V'$, the preimage $\Phi^{-1}(U')$ is $\beta$-coisotropic if and only if 
$U'$ is $\beta'$-coisotropic. If $U\subseteq V$ is $\beta$-coisotropic, then $\Phi(U)\subseteq V'$ is $\beta'$-coisotropic. (See \cite[Lemma 6.1.2]{rob:cla}.)

\subsubsection{} 
Let $\dd$ be a Lie algebra, together with an $\ad$-invariant element $\beta\in S^2\dd$. 
We denote the co-adjoint action on $\dd^*$ simply by a bracket, so that 
$\l [\lambda_1,\mu],\ \lambda_2\r=-\l \mu,[\lambda_1,\lambda_2]\r$ for $\lambda_1,\lambda_2\in\dd$ and $\mu\in\dd^*$. Denote by $\dd^*_\beta$ the vector space $\dd^*$, equipped with the Lie bracket
$[\mu_1,\mu_2]_\beta:=[\beta^\sharp(\mu_1),\mu_2]$. 
The coadjoint action of $\dd$ is by derivations of this Lie bracket, 
and the semi-direct product $\wt\dd=\dd\ltimes \dd^*_\beta$ becomes a metrized Lie algebra for the metric 
\begin{equation}
\l \lambda_1+\mu_1,\lambda_2+\mu_2\r=\l\lambda_1,\mu_2\r+\l\lambda_2,\mu_1\r+\beta(\mu_1,\mu_2) 
\end{equation}
%\[  \begin{split}\l \zeta_1+\mu_1,\zeta_2+\mu_2\r_\beta&=\l\zeta_1+\mu_1,\zeta_2+\mu_2\r_0+\beta(\mu_1,\mu_2) ,\\[\zeta_1+\mu_1,\zeta_2+\mu_2]_\beta&=[\zeta_1+\mu_1,\zeta_2+\mu_2]_0+\ad(\beta^\sharp(\mu_1))(\mu_2),\end{split}\]
for $\lambda_i\in\dd,\ \mu_i\in\dd^*$. We will denote by $\wt{\beta}\in S^2\wt\dd$ the element dual to the metric. 
Let $\sz_{\wt{\dd}},\tz_{\wt{\dd}}\colon \dd\ltimes \dd^*_\beta\to \dd$ be the 
maps 
\[ \sz_{\wt{\dd}}(\lambda+\mu)=\lambda,\ \ \ \tz_{\wt{\dd}}(\lambda+\mu)=\lambda+\beta^\sharp(\mu).\]
These are the source and target maps for its structure as an action groupoid $\dd\ltimes \dd^*_\beta\rra \dd$, for the action $\lambda\mapsto \lambda+\beta^\sharp(\mu)$. 
One has 
\[ \sz_{\wt{\dd}}(\wt{\beta})=-\beta,\ \ \tz_{\wt{\dd}}(\wt{\beta})=\beta.\]
%
%As explained in  \cite{lib:dir}, the pairs $(\dd,\beta)$ classify Dirac Lie group structures over $\pt$, in such a way that the graph of the Dirac multiplication morphism coincides with the graph of the groupoid multiplication. 
%The Lie subalgebra $\dd\subseteq \wt{\dd}$ is Lagrangian, and has  $\dd^*_\beta$ as a complementary Lie algebra ideal. The orthogonal space $(\dd^*_\beta)^\perp$ is again a complementary ideal.  Letting $f_{\wt{\dd}}\colon \wt{\dd}\to \dd$ be the Lie algebra morphism with kernel $(\dd^*_\beta)^\perp$, 
%The dual map  is just the given inclusion of $\dd^*$ as $\dd^*_\beta\subseteq \wt{\dd}\cong  \wt{\dd}^*$. 
%Explictly, $f_{\wt{\dd}}(\lambda+\mu)=\lambda+\beta^\sharp(\mu)$.
%

\subsubsection{}\label{subsubsec:LApair}
A  \emph{Dirac-Manin pair}  $(\dd,\g)_\beta$ is a Lie algebra $\dd$ together with 
an $\ad$-invariant element $\beta\in S^2\dd$ and a $\beta$-coisotropic Lie subalgebra $\g\subseteq \dd$. If $\beta$ is non-degenerate (so that it defines a metric on $\dd$), and $\g$ is Lagrangian, then one calls  $(\dd,\g)_\beta$ a \emph{Manin pair}.

Given a Dirac-Manin pair $(\dd,\g)_\beta$, the pre-image  $\sz_{\wt{\dd}}^{-1}(\g)=\g\ltimes \dd^*_\beta\subseteq  \wt{\dd}$ is a
$\wt{\beta}$-coisotropic Lie subalgebra. The orthogonal space is an ideal in 
$\sz_{\wt{\dd}}^{-1}(\g)$, hence 
\[ \q=(\g\ltimes \dd^*_\beta)/(\g\ltimes \dd^*_\beta)^\perp\]
is a metrized Lie algebra. Let $\gamma_\q\in S^2\q$ be given by the reduced metric on $\q$, and let $\g\subseteq \q$ be embedded as the reduction of $\dd\subset \wt{\dd}$. 
Then $(\q,\g)_{\gamma_\q}$ is a Manin pair. The map $\tz_{\wt{\dd}}$ descends to a Lie algebra morphism 
\[ f_\q\colon \q\to \dd,\] 
restricting to the inclusion on $\g\subseteq \q$, and with $f_\q(\gamma_\q)=\beta$. 
%One thus obtains a Dirac morphism $(\dd\ltimes \dd^*_\beta,\dd)_{\wt{\beta}}\da (\q,\g)_{\gamma_\q}$. 

\subsubsection{}\label{subsec:hc} {\it Harish-Chandra pairs, dressing action.}
Let $\dd$ be a Lie algebra, and $H$ a Lie group acting on $\dd$ by Lie algebra automorphisms. An $H$-equivariant inclusion 
of the Lie algebra $\h$  is said to define \emph{generators} if the differential of the $H$-action on $\dd$ coincides with the adjoint action of $\h\subseteq \dd$. In this case, we will denote the $H$-action on $\dd$ simply by $h\mapsto \Ad_h$, and call $(\dd,H)$  a \emph{Harish-Chandra pair}.
A \emph{morphism of Harish-Chandra pairs} $(\cc,K)\to (\dd,H)$ is a Lie group morphism $\phi\colon K\to H$, together with a Lie algebra morphism $f\colon \cc\to \dd$, such that 
$f$ is $K$-equivariant and satisfies $f|_\k=T_e\phi\colon \k\to \h$.

Given a Lie algebra $\g$ complementary to $\h$, we will call $(\dd,\g,\h)$ an \emph{$H$-equivariant Lie algebra triple}. In this situation one obtains a Lie algebra action of $\dd$ 
on $H$, extending the action $\tau\mapsto \tau^L$ of $\h$ given by the left-invariant vector fields. Denote by $\pr_\g$ and $\pr_\h$ the projections from $\dd$ onto the two summands. In terms of left trivialization $TH=H\times \h$, the action 
$\varrho\colon \dd\to \Gamma(TH)$ is given by 
\begin{equation} \label{eq:action}
\varrho(\lambda)_h=\Ad_{h^{-1}}\pr_\h(\Ad_h \lambda),\ \ \ \lambda\in \dd.\end{equation}
In the context of Poisson Lie groups, this action 
is known as the (right) \emph{dresssing} action. (There is also a left dressing action $(h,\lambda)\mapsto -\pr_\h(\Ad_{h^{-1}}\lambda)$, but we will not use it. )

\subsubsection{}\label{subsub:vacantLA}
A \emph{Dirac-Manin triple} $(\dd,\g,\h)_\beta$ is a Dirac-Manin pair $(\dd,\g)_\beta$ together with a Lie subalgebra 
$\h\subseteq  \dd$ complementary to $\g$. It is called  a 
\emph{Manin triple} if $(\dd,\g)_\beta$ is a Manin pair and 
$\h$ is Lagrangian. A Dirac Manin triple $(\dd,\g,\h)_\beta$ with an action of $H$ on $\dd$ is called 
an \emph{$H$-equivariant Dirac-Manin triple} if $(\dd,H)$ is a Harish-Chandra pair, and 
$\beta$ is $H$-invariant.

\subsection{Courant algebroids}\label{subsubsec:courant}
A \emph{Courant algebroid} \cite{liu:ma} is a vector bundle $\AA\to M$, equipped with a non-degenerate symmetric fiberwise bilinear form 
$\l\cdot,\cdot\r$ called the \emph{metric}, a bundle map $\mathsf{a}\colon \AA\to TM$ called the \emph{anchor}, and a bilinear map 
$\Cour{\cdot,\cdot}\colon \Gamma(\AA)\times \Gamma(\AA)\to \Gamma(\AA)$ 
called the \emph{Courant bracket}, satisfying  the following axioms, 
for all sections $\sig_1,\sig_2,\sig_3\in \Gamma(\AA)$:
\begin{enumerate}
\item $\Cour{\sig_1,\Cour{\sig_2,\sig_3}}=\Cour{\Cour{\sig_1,\sig_2},\sig_3}+
\Cour{\sig_2,\Cour{\sig_1,\sig_3}}$, 
\item $\mathsf{a}(\sig_1)\l \sig_2,\sig_3\r=\l \Cour{\sig_1,\sig_2},\sig_3\r +\l \sig_1, \Cour{\sig_2,\sig_3} \r$, 
\item $\Cour{\sig_1,\sig_2}+\Cour{\sig_2,\sig_1}=\mathsf{a}^* \mbox{d} \l\sig_1,\sig_2\r$. 
\end{enumerate}
(In the last condition, the metric is used to identify $\AA^*$ with $\AA$.) The bracket is sometimes also called the \emph{Dorfman bracket}, after \cite{dor:dir}. The conditions imply \cite{uch:rem} that 
the anchor map is bracket-preserving, and that 
\begin{enumerate}
\item[(d)]  $\Cour{\sigma_1,f\sigma_2}=f\Cour{\sigma_1,\sigma_2}+(\L_{\mathsf{a}(\sigma_1)}f) \sigma_2$
\end{enumerate}
for $f\in C^\infty(M)$.
Note that a Courant algebroid over $M=\pt$  is the same thing as a metrized Lie algebra. For a Courant algebroid
$\AA$, changing the sign of the metric defines a Courant algebroid $\ol{\AA}$. 
The \emph{standard Courant algebroid} over $M$ is the direct sum $\T M=TM \oplus T^* M$, with the metric given by the 
pairing between vectors and covectors, with anchor the projection to the first summand,  and with the Courant bracket 
\[ \Cour{X_1+\alpha_1,\,X_2+\alpha_2}=[X_1,X_2]+\L(X_1)\alpha_2-\iota(X_2)\mbox{d}\alpha_1\]
for vector fields $X_1,X_2\in\Gamma(T M)$ and 1-forms $\alpha_1,\alpha_2\in\Om^1(M)$.  
This is essentially the setting of Courant's original work \cite{cou:di,couwein:beyond}.

\subsection{Dirac manifolds}
A \emph{Dirac structure} \cite{liu:ma} on a manifold $M$ is a Courant algebroid $\AA$ over $M$ together with a 
Lagrangian subbundle $E\subseteq \AA$ that is \emph{involutive}, that is, $\Gamma(E)$ is closed under the Courant bracket. 
For any Dirac structure, the Courant bracket restricts to a Lie bracket on $\Gamma(E)$, making $E$ into a Lie algebroid. 
The triple $(M,\AA,E)$ will be called a \emph{Dirac manifold}. We often refer to $M$ itself as a Dirac manifold, and to  
$(\AA,E)$ as its Dirac structure. 
%Given two Dirac manifolds, one can form their product; the trivial Dirac manifold $(\pt,0,0)$ serves as a neutral element for this product.
For the standard Courant algebroid, the Dirac structures $(\T M,\,E)$ with the property $E\cap TM=0$ are in 1-1 correspondence with Poisson structures $\pi_M$; the correspondence takes the Poisson structure to $E=\on{Gr}(\pi_M)$, the graph of the bundle map $\pi_M^\sharp\colon T^*M\to TM,\ \alpha\mapsto \pi_M(\alpha,\cdot)$. 

\subsection{The category $\ca{DIR}$}\label{subsec:dircat}
For any Courant algebroid
$\AA$, we denote by $\ol{\AA}$ the Courant algebroid obtained by 
changing the sign of the metric. 
A \emph{Courant morphism}
$R\colon \AA_1\da \AA_2$ \cite{al:der,bur:cou} between two Courant algebroids is a smooth map $\Phi\colon M_1\to M_2$ 
of the underlying manifolds, together with a Lagrangian subbundle $R\subseteq  \AA_2\times \ol{\AA}_1$ 
along the graph $\on{Gr}(\Phi)$, such that $(\mathsf{a}_2\times\mathsf{a}_1)(R)\subseteq  \on{Gr}(T\Phi)$, and satisfying 
an integrability condition: if two sections of $\AA_2\times \ol{\AA}_1$ restrict over $\on{Gr}(\Phi)$ to a section of 
$R$, then so does their Courant bracket. We shall write $x_1\sim_R x_2$ if and only if $(x_2,x_1)\in R$.  Since $R$ is Lagrangian, we have 
\begin{equation}\label{eq:-1}
x_1\sim_R x_1',\ x_2\sim_R x_2'\ \ \ \Rightarrow \ \ \ \l x_1,x_1'\r=\l x_2,x_2'\r.
\end{equation}
For sections $\sig_i\in \Gamma(\AA_i)$, write $\sig_1\sim_R\sig_2$ if and only if $(\sig_2\times\sig_1)|_{\on{Gr}(\Phi)}\in \Gamma(R)$. The integrability of $R$ implies 
\begin{equation}\label{eq:0}
\sig_1\sim_R \sig_2,\ \sig_1'\sim_R \sig_2'\ \ \ \Rightarrow \ \ \ 
\Cour{\sig_1,\sig_1'}\sim_R \Cour{\sig_2,\sig_2'}.\end{equation}
Composition of Courant morphisms is defined as a composition of relations, provided that 
certain transversality conditions are satisfied (cf.~ \cite{lib:dir}). 
 A \emph{Dirac morphism} \cite{bur:cou}
\begin{equation}\label{eq:dmor} (\Phi,R)\colon (M_1,\AA_1,E_1)\da (M_2,\AA_2,E_2)\end{equation}
is  a Courant morphism $R\colon \AA_1\da \AA_2$,  with base map 
$\Phi\colon  M_1\to M_2$, with the following property: 
\begin{quote}($\ca{D}$)\ \ For all $m\in M$ and every element $x_2\in (E_2)_{\Phi(m)}$, there is a \emph{unique} element $x_1\in (E_1)_m$ with $x_1\sim_R x_2$. 
\end{quote}
Given another Dirac morphism 
$(\Phi',R')\colon (M_2,\AA_2,E_2)\da (M_3,\AA_3,E_3)$, the composition $(\Phi'\circ \Phi,R'\circ R)$ is a 
well-defined Dirac morphism $(M_1,\AA_1,E_1)\da (M_3,\AA_3,E_3)$ -- the transversality conditions 
mentioned above are automatic due to the uniqueness condition in ($\ca{D}$). We hence have a well-defined \emph{category of Dirac manifolds}. Any Dirac morphism induces a \emph{comorphism} $E_1\da E_2$ of Lie algebroids: A bundle map $\Phi^*E_2\to E_1$ such that the induced map on sections is a Lie algebra morphism, and is compatible with the anchor. 
For any map $\Phi\colon M_1\to M_2$, the direct sum of the graphs of the differential $T\Phi$ and of its dual $(T\Phi)^*$ defines a Courant morphism 
\[\T \Phi\colon \T M_1\da \T M_2.\] 
If $M_1,M_2$ are Poisson manifolds, then $\Phi$ is a Poisson map if and only if it defines a Dirac morphism
$\T \Phi\colon (M_1,\T M_1,\on{Gr}(\pi_1))\da (M_2,\T M_2,\on{Gr}(\pi_2))$. 
We will use the following 
notation for the categories that we are interested in: 
\[ \begin{split}
\ca{POI} &  \mbox{ -- Poisson manifolds and Poisson maps}\\
\ca{DIR} &  \mbox{ -- Dirac manifolds and Dirac morphisms }\\
\ca{LA}^\vee & \mbox{ -- Lie algebroids and comorphisms of Lie algebroids}
\end{split}
\]
We have a functor $\ca{POI} \to \ca{DIR}$ taking $(M,\pi)$ to $(M,\T M,\on{Gr}(\pi))$, 
and a functor $\ca{DIR}\to \ca{LA}^\vee$ taking $(M,\AA,E)$ to $(M,E)$. By composition, one obtains the functor $\ca{POI}\to \ca{LA}^\vee$ associating to a Poisson manifold its cotangent Lie algebroid $\on{Gr}(\pi)\cong T^*_\pi M$. There is also  a functor $\ca{LA}^\vee\to \ca{POI}$ in the opposite direction, taking a Lie algebroid $E$ to the dual bundle $E^*$ with a linear Poisson structure on its total space; as is well-known, the composition gives the \emph{tangent lift} of a Poisson structure from $M$ to $TM$. 

One also encounters \emph{Dirac comorphisms} (also called \emph{backward Dirac} \cite{bur:di}): These are defined similar to Dirac morphisms, but with ($\ca{D}$) replaced by the condition 
\begin{quote}
($\ca{D}^\vee$)\ \ For all $m\in M$ and every $x_1\in (E_1)_m$, there is a unique $x_2\in (E_2)_{\Phi(m)}$ such that $x_1\sim_R x_2$.  
\end{quote}
This defines a category $\ca{DIR}^\vee$ with a functor to the category $\ca{LA}$ of Lie algebroids and Lie algebroid morphisms.

\subsection{Action Dirac structures}\label{subsec:ads}
Let $\q$ be a metrized Lie algebra, with metric given by an $\ad$-invariant element $\gamma\in S^2\q$. 
Given an action $\varrho\colon \q\to \Gamma(TM)$ on a manifold $M$,  with the property that all stabilizer algebras 
for the action are coisotropic in $\q$, the trivial bundle $M\times\q$ has a well-defined 
structure of an  \emph{action Courant algebroid} \cite{lib:cou}. The anchor $\mathsf{a}\colon M\times \q\to TM$ coincides with the action map, 
and the metric and Courant bracket extend the  metric and Lie bracket on constant sections. 
Any Lagrangian Lie subalgebra $\g\subseteq \q$ defines an \emph{action Dirac structure} 
\begin{equation}\label{eq:triple}
 (M\times\q,\ M\times\g)
 \end{equation}
on $M$. As an example, consider any Lie algebra $\g$, and take $\q=\g\ltimes\g^*$ with the metric given by the pairing. 
Its natural action on $M=\g^*$ has coisotropic stabilizers, hence $M=\g^*$ becomes a Dirac manifold. In this example, 
the Dirac structure is canonically isomorphic to $(\T \g^*,\on{Gr}(\pi_{\g^*}))$ with the usual Lie-Poisson structure 
$\pi_{\g^*}$.

\subsection{Actions of Lie algebroids and Courant algebroids}\label{subsec:actLC}
An \emph{action of a Lie algebroid $E\to M$ along a map $\Phi\colon N\to M$} is a
comorphism of Lie algebroids 
\[ \xymatrix{ TN \ar@{-->}[r]\ar[d] & E\ar[d]\\ N\ar[r]_{\Phi} & M
}\]
Given such an action, the vector bundle pull-back $\Phi^* E$ 
has a Lie algebroid structure, in such a way that the natural map  $\Phi^*E\to E$
is a Lie algebroid comorphism. Similarly \cite{lib:cou}, an action of a Courant algebroid $\AA\to M$ on $N$ is given by a vector bundle comorphism $TN\da \AA$, such that the induced map $\varrho_N\colon \Gamma(\AA)\to \Gamma(TN)$ preserves brackets and is compatible with the anchor: 
\[ \varrho_N(\Cour{\sigma_1,\sigma_2})=[\varrho_N(\sigma_1),\varrho_N(\sigma_2)],\ \ \ T\Phi\circ \varrho_N=\mathsf{a}_\AA.\]
If furthermore the stabilizers of this action (i.e., the kernels of $\AA_{\Phi(n)}\to T_nN$)
are coisotropic in $\AA$, then 
the vector bundle pullback $\Phi^*\AA$ has a Courant algebroid structure, 
and comes with a morphism of Courant algebroids $\Phi^*\AA\da \AA$. 
In the case of a Dirac structure $(\AA,E)$ this becomes a Dirac morphism 
$(\Phi^*\AA,\Phi^*E)\da (\AA,E)$. 

\subsection{Maps compatible with brackets}\label{subsec:compatible}
Let $\AA$ be a Courant algebroid with base $M$, and let $\dd$ be a Lie algebra. A bundle map 
\[ f\colon \AA\to \dd\]
(with base map $M\to \pt$) will be called \emph{compatible with brackets} if the map 
$(\mathsf{a}_\AA,f)\colon \Gamma(\AA)\to \Gamma(TM\times \dd)$ is bracket preserving, 
where $TM\times \dd$ has the product Lie algebroid structure. That is, for all  $\sigma_1,\sigma_2\in\Gamma(\AA)$, 
\begin{equation}\label{eq:bracketpreserving}
 f(\Cour{\sigma_1,\sigma_2})=[f(\sigma_1),f(\sigma_2)]
+\L_{\mathsf{a}_\AA(\sigma_1)}f(\sigma_2)-\L_{\mathsf{a}_\AA(\sigma_2)}f(\sigma_1).
\end{equation}
Note that since the right hand side is skew-symmetric in $\sigma_1,\sigma_2$, this equation means in particular that 
\begin{equation}\label{eq:rema}
 f\circ \mathsf{a}_\AA^*=0,\end{equation}
by property (c) in the definition of Courant algebroids. 

Suppose in addition that $\dd$ has an $\ad$-invariant element $\beta\in S^2\dd$, and that 
$f(\gamma_\AA)=\beta$, where $\gamma_\AA\in \Gamma(S^2\AA)$ is the section 
determined by the metric. Together with \eqref{eq:rema}, this means $(\mathsf{a}_\AA,f)(\gamma_\AA)=\beta$, viewed as a section of $S^2(TM\times \dd)$. Suppose $\varrho_Q\colon \dd\to \Gamma(TQ)$ is an action of $\dd$ on a manifold $Q$, with $\beta$-coisotropic stabilizers. By composition, we obtain an action of $\AA$ on $Q\times M$ along $\pr_M$, defined by the comorphism
\[ T(Q\times M) \da \dd\times TM \da \AA\]
(where we regard the morphism $(\mathsf{a}_\AA,f)$ as a comorphism in the opposite direction). That is, 
$ \varrho_{Q\times M}(\sigma)=\varrho_Q(f(\sigma))+\mathsf{a}_\AA(\sigma)$. 
By Section \ref{subsec:actLC}, it follows that $\pr_M^*\AA=Q\times \AA$ is a Courant algebroid over $Q\times M$.  If $(\AA,E)$ is a Dirac structure, then we obtain a Dirac structure $(Q\times \AA,\,Q\times E)$. The action Dirac structures from Section \ref{subsec:ads} may be regarded as such pullbacks. See \cite[Section 4]{lib:cou}.

\subsection{Pull-backs}\label{subsec:pull1}
There is another kind of pullback  for Lie algebroids, Courant algebroids, and Dirac structures, as follows. Suppose $\Phi\colon N\to M$ is a smooth map, and $E$ is a Lie algebroid over $M$. Suppose that  the anchor $\mathsf{a}_E$ 
is transverse to $T\Phi$. Then the fiber product 
\[ \xymatrix{ \Phi^!E \ar[r]\ar[d] & E\ar[d]^{\mathsf{a}_E} \\ TN \ar[r]_{T\Phi} & TM
}
\]
is a vector bundle $\Phi^!E\to N$, with a canonical Lie algebroid structure \cite{mac:gen} (cf.~ \cite{lib:cou}). Here the left vertical map serves as the anchor, and the upper horizontal map 
defines a Lie algebroid morphism to $E$. Similarly, if $\AA\to M$ is a Courant algebroid whose anchor $\mathsf{a}_\AA$ is transverse 
to $T\Phi$, then one obtains a Courant algebroid $\Phi^!\AA=C/C^\perp$ where $C$ is the fiber product of $\AA$ with $\T N$ over $TM$ 
\cite{lib:dir}. Finally, given a Dirac structure $(\AA,E)$ over $M$ such that $\mathsf{a}_E$ is transverse to $T\Phi$, we obtain a \emph{pull-back Dirac structure} \cite{lib:cou} $(\Phi^!\AA,\, \Phi^! E)$ on $N$, with a Dirac \emph{co}morphism 
\[ (\Phi,R)\colon (N,\Phi^!\AA,\Phi^!E)\da (M,\AA,E).\]
%Here $\Phi^!E=E\cap C/E\cap C^\perp$ is isomorphic to the pullback as a Lie algebroid. 
If $N$ is a direct product $N=Q\times M$, and $\Phi$ is projection to the second factor, then 
$\Phi^!\AA=\T Q\times \AA$.

As another special case, $\Phi^!(\T M)=\T N$. If $E\subseteq  \T M$ is a Dirac structure, then $\Phi^!E$ is a Poisson structure if and only if $E$ is a Poisson structure and $\Phi$ is an immersion.

\subsection{Reductions}
Let $(M,\AA,E)$ be a $K$-equivariant Dirac manifold. That is, $K$ is a Lie group  acting on $M$, with a lift to an action by 
Courant automorphisms of $\AA$, preserving the subbundle $E$. For $\nu\in\k$ let $\L_\AA(\nu)$ be the 
resulting Lie derivatives on $\Gamma(\AA)$. 
A $K$-equivariant bundle map 
\[ \varrho_\AA\colon M\times \k\to \AA\]
is said to define \emph{generators} if $\L_\AA(\nu)=\Cour{\varrho_\AA(\nu),\cdot}$ for all $\nu\in\k$. We will only 
consider the case that the generators are isotropic, that is,  $\l\varrho_\AA(\nu),\varrho_\AA(\nu)\r=0$ for all $\nu$. 
Suppose the $K$-action is a principal action, and let $J=\varrho_\AA(M\times\k)$. 
One defines a \emph{reduced Courant algebroid}  $\AA\qu K=(J^\perp/J)/K$. Assuming that $E\cap J$ has constant rank (or equivalently, 
$E\cap J^\perp$ has constant rank), one defines $E\qu K=\big((E\cap J^\perp)/(E\cap J)\big)/K$. Then $(\AA\qu K,\  E\qu K)$ is a Dirac structure on $M/K$. 
See Bursztyn, Cavalcanti, and Gualtieri \cite{bur:cou} for details and much more general reduction procedures; see also 
Marrero, Padron, and Rodrigues-Olmos
\cite{mar:red} for
a detailed discussion of reductions of Lie algebroids. 
If $E\cap J=0$ (resp., $J\subseteq E$) then the quotient map $p\colon M\to M/K$ 
lifts to a Dirac morphism (resp. \emph{co}morphism)  
\[ (M,\AA,E)\da (M/K,\  \AA\qu K,\  E\qu K).\] 
%In particular, if $E$ is the graph of a $K$-invariant Poisson structure, and taking for $\varrho_\AA$ the generating vector fields viewed as sections of $TM$, then $E\qu K$ is the graph of the induced Poisson structure on $M/K$.
The pull-back operation $p^!$ is a right inverse to reduction, in the following 
sense:
\begin{proposition}
Let $p\colon M\to B$ be a principal $K$-bundle, and $(\PP,\,L)$ a Dirac structure over the 
base $B$. Then the pull-back Dirac structure $(p^! \PP,\, p^!L)$ on $M$ 
is $K$-equivariant, with isotropic generators, and its reduction by $K$ is canonically isomorphic to 
$(\PP,\,L)$. 
\end{proposition}
\begin{proof}
Let $C\subseteq  \PP\times \T M$ be the coisotropic subbundle along $\on{Gr}(p)$, given as 
the fiber product over $TB$.
By definition, 
\[ p^!\PP=C/C^\perp,\ \ p^!L=(L\times TM)\cap C\,\big/(L\times TM)\cap C^\perp.\]
For $(y,z)\in C$, let $[(y,z)]\in p^!\PP$ be its equivalence class.  The 
Lagrangian subbundle $R\subseteq  \PP\times \ol{p^!\PP}$, consisting of all 
$(y,[(y,z)])$ with $(y,z)\in C$, defines the 
Courant morphism $R\colon p^!\PP\da \PP$. Consider the canonical action of $K$ on $\T M$, with generators $\k\to \Gamma(TM)\subseteq\Gamma(\T M)$
given by the action on $M$. Its direct product with the \emph{trivial} action on $\PP$ 
preserves $C$, and its generators descends to generators $\k\to \Gamma(C/C^\perp)$. Let $J\cong M\times\k$ be the 
isotropic subbundle of $C/C^\perp$ spanned by the generators. Then $J\subseteq p^!L\subseteq C/C^\perp$, and the reduction $(p^! L)\qu K \subseteq (p^!\PP)\qu K$ is defined. Since elements $x\in J$ satisfy $x\sim_R 0$, the morphism 
$R$ descends to a 
Dirac morphism 
\[ (\on{id}_B,\,R\qu K)\colon \big(B,\,(p^!\PP)\qu K,\,(p^!L)\qu K\big)\da (B,\,\PP,\,L).\]
Since any element $y\in \PP$ satisfies $[(y,v)]\sim_R y$, where 
$v\in TM$ is any lift of $\mathsf{a}_{\PP}(y)\in TB$, the morphism $R\qu K$ is surjective. Hence, by dimension considerations it is an isomorphism.
\end{proof}

\section{Dirac Lie groups and Dirac actions}
\subsection{Definitions}
Poisson Lie groups do not conform to the official definition of \emph{group objects} in the category $\ca{POI}$ of Poisson manifolds: The product of Poisson manifolds is not a direct product in the categorical sense, and in any case the inverse map is anti-Poisson rather than Poisson. See Blohmann and Weinstein \cite{blo:gro} for a detailed discussion. What makes them `group-like' objects is that they come with an associative  \emph{multiplication morphism} and a \emph{unit morphism}, satisfying the usual properties. (The existence of an 
inverse, and the fact that it is anti-Poisson,  are automatic.)  In a similar spirit, we can 
define \emph{Dirac Lie groups} to be `group-like' objects in the category $\ca{DIR}$ of Dirac manifolds. 
For a Lie group $H$, we denote by $\Mult_H\colon H\times H\to H$ the multiplication map, and by 
$\eps_H\colon \pt\to H$ the unit map (that is, the inclusion of the group unit).
\begin{definition}\label{def:dir}
\begin{enumerate}
\item\label{it:aa1}
A \emph{Dirac Lie group} is a Lie group $H$ with a Dirac structure $(\AA,E)$, together with Dirac morphisms 
\[ \Mult_\AA\colon (H,\AA,E)\times (H,\AA,E)\da (H,\AA,E)\] 
and 
\[ \eps_\AA\colon (\pt,0,0)\da (H,\AA,E),\]
with base maps $\Mult_H$ and $\eps_H$ respectively, 
 such that $\Mult_\AA$ is associative and $\eps_\AA$ is a two-sided unit:
\[ \Mult_\AA\circ (\Mult_\AA\times \id_\AA)=\Mult_\AA\circ (\id_\AA\times \Mult_\AA),\]
\[ \Mult_\AA\circ (\eps_\AA\times \id_\AA)=\id_\AA=\Mult_\AA\circ (\id_\AA\times \eps_\AA).\]
\item\label{it:aa2}
A \emph{Dirac action}  of $(H,\AA,E)$ on a Dirac manifold $(M,\PP,L)$ is a Dirac morphism 
\[\ca{A}_\PP\colon (H,\AA,E)\times (M,\PP,L)\da (M,\PP,L),\]
with base map $\mathsf{a}_M\colon H\times M\to M$ an $H$-action, satisfying 
\[ \begin{split}
\ca{A}_\PP\circ (\on{id}_\AA\times \ca{A}_\PP)&=\ca{A}_\PP\circ (\Mult_\AA\times \on{id}_\PP),\\
\ca{A}_\PP\circ (\eps_\AA\times \id_\PP)&=\id_\PP.
\end{split}\]
\end{enumerate}
\end{definition}
The following result shows that Dirac actions can always be described as actions of 
$\VB$-groupoids. We refer to Appendix \ref{app:vb} for background and notation for 
$\VB$-groupoids and their actions.

\begin{theorem}\label{th:ingroupoid}
\begin{enumerate}
\item\label{it:1a}
For any Dirac Lie group $(H,\AA,E)$, the Courant algebroid $\AA$ is naturally a $\VB$-groupoid  $\AA\rra \AA^{(0)}$ with base the group $H\rra \pt$, and with units 
$\AA^{(0)}=E_e$,  
in such a way that $\Mult_\AA$ becomes the groupoid multiplication. The subbundle $E$ is a vacant 
$\VB$-subgroupoid,  with the same space of units $E^{(0)}=E_e$. 
\item \label{it:1b}
For a Dirac action of $(H,\AA,E)$ on a Dirac manifold $(M,\PP,L)$ the bundle $\PP$ is naturally a 
$\VB$-groupoid module over $\AA$, in such a way that $\mathsf{a}_\PP$ becomes the groupoid module action. 
The subbundle $L$ is a submodule over the subgroupoid $E$.
\end{enumerate}
\end{theorem}
\begin{proof}
\eqref{it:1a} For the proof, we need to define the source and target maps, verify that $\Mult_\AA$ gives a well-defined 
multiplication of composable elements, and finally prove the existence of groupoid inverses. 
\begin{enumerate}
\item[1.] {\it Definition of source and target map.} Let $\g:=E_e$. We can think of $\g$ as the 
Lagrangian subbundle of $\AA\times 0$ along $\on{Gr}(\eps_H)\subseteq H\times\pt$
defining the Dirac morphism
$\eps_\AA$. The identity $\Mult_\AA\circ (\eps_\AA\times \id_\AA)=\id_\AA$ says that 
for any given $x\in \AA$ there exists $\xi\in \g$ with 
\[ x\sim_{(\eps_\AA\times \id_\AA)}(\xi,x)\sim_{\Mult_\AA} x.\]
This $\xi$ is unique, since $(\xi,0_h)\sim_{\Mult_\AA} 0_h$ implies $\xi=0$, by definition of Dirac morphisms.
We take it to be the definition of $\tz(x)$, so that $(\tz(x),x)\sim_{\Mult_\AA} x$. Similarly, there is a unique 
element $\sz(x)\in\g$ such that $(x,\sz(x))\sim_{\Mult_\AA} x$. 
\item[2.] We next prove that 
\[ (x_1,x_2)\sim_{\Mult_\AA} x \Rightarrow \sz(x_1)=\tz(x_2),\ \tz(x_1)=\tz(x),\ \sz(x_2)=\sz(x).\]
The relation $(x_1,x_2)\sim_{\Mult_\AA} x$ gives
\[ (x_1,\sz(x_1),x_2)\sim_{\Mult_\AA\times \on{id}} (x_1,x_2)\sim_{\Mult_\AA} x;\]
hence by associativity of $\Mult_\AA$ there exists $y$ such that 
\[ (x_1,\sz(x_1),x_2)\sim_{\on{id}\times \Mult_\AA} (x_1,y)\sim_{\Mult_\AA} x.\]
In particular, $(\sz(x_1),x_2)\sim_{\Mult_\AA} y$. Subtracting $(\tz(x_2),x_2)\sim_{\Mult_\AA} x_2$, this shows 
\[ (\sz(x_1)-\tz(x_2),0)\sim_{\Mult_\AA} y-x_2.\]
Since the left hand side is in $E\times E$, it follows that $z:=y-x_2\in E$.  But we also have $(\tz(z),z)\sim z$; hence the uniqueness 
condition in the definition of Dirac morphism shows $(\tz(z),z)=(\sz(x_1)-\tz(x_2),0)$. This proves $z=0$ and hence $\sz(x_1)=\tz(x_2)$. 
By a similar argument, $\tz(x_1)=\tz(x)$ and $\sz(x_2)=\sz(x)$. 
\item[3.] {\it The maps $\sz$ and $\tz$ restrict to fiberwise isomorphisms on $E$.}
Let $h\in H$ be given. For any $\xi\in\g$ there exist, by definition of Dirac morphism, elements $x\in E_h$ and $y\in E_{h^{-1}}$ with 
$(x,y)\sim_{\Mult_\AA} \xi$. We have $\tz(x)=\xi=\sz(y)$. This shows that both $\sz$ and $\tz$ restrict to 
fiberwise isomorphisms on $E$. It also follows that $\ker(\sz)$ and $\ker(\tz)$ 
are subbundles of $\AA$, both of which are complements to $E$ in $\AA$.
\item[4.] {\it  Definition of the groupoid multiplication $\circ$.} 
If $(0_{h_1},0_{h_2})\sim_{\Mult_\AA} z$, then $z\in E_{h_1h_2}$ 
by definition of Dirac morphism, but also $\sz(z)=\sz(0_{h_2})=0$, hence $z=0$. This shows that if $x_1,x_2\in \AA$ are such that 
$(x_1,x_2)\sim_{\Mult_\AA} x$ for some $x\in \AA$, then this $x$ is unique. 
In this case, we will call $x_1,x_2$ \emph{composable}, and write 
$x=x_1\circ x_2$.  As we saw above, a necessary condition for composability is that 
$\sz(x_1)=\tz(x_2)$. Since the subbundle consisting of elements  $(x_1,x_2)\in \AA\times \AA$ with this property has rank
$2\on{rank}(\AA)-\dim\g=3\dim\g=\on{rank}(\on{Gr}(\Mult_\AA))$, it follows that 
this condition is also sufficient. 
\item[5.] {\it  Existence of a groupoid inverse.}
Let $x\in \AA_h$ be given. Let $y_1\in E_{h^{-1}}$ be the unique element with $\sz(x)=\tz(y_1)$, 
and put $y=y_1+0_{h^{-1}}\circ \lambda$ with $\lambda=\tz(x)-x\circ y_1\in \ker(\tz_e)\subseteq \AA_e$. Then
\[ x\circ y=x\circ y_1+0_h\circ 0_{h^{-1}}\circ \lambda=x\circ y_1+\lambda=\tz(x),\]
so that $y$ is a right inverse to $x$. Similarly, one has the existence of a left inverse. By the usual argument, it is automatic that 
left and right inverses coincide. 
\end{enumerate}

\eqref{it:1b}
Arguing as in part \eqref{it:1a}, we see that for any $y\in \PP$, there exists 
a unique element $\uz(y)\in \g$ with $(\uz(y),y)\sim_{\ca{A}_\PP} y$. Furthermore, $(x,y)\sim_{\ca{A}_\PP} y'$ implies that 
$\sz(x)=\uz(y)$ and $\tz(x)=\uz(y')$.  

We claim that $(x,y)\sim_{\ca{A}_\PP} y'$ uniquely determines $y'$ in terms of $x,y$. Indeed, suppose $w\in\PP_{h.m}$ with 
$(0_h,0_m)\sim_{\ca{A}_\PP} w$. We will show $w=0$ by proving that $\l w,z\r=0$ for all  $z\in \PP_{h.m}$.
Write $\uz(z)=y_1\circ y_2$ with $y_1\in E_h$ and $y_2\in E_{h^{-1}}$. Then 
\[ (y_1,y_2,z)\sim_{\Mult_\AA\times\id_\PP} (\uz(z),z)\sim_{\mathsf{a}_\PP} z.\]
By the property $\ca{A}_\PP\circ (\on{id}_\AA\times \ca{A}_\PP)=\ca{A}_\PP\circ (\Mult_\AA\times \on{id}_\PP)$, there exists $z'\in \PP_m$ with 
\[ (y_1,y_2,z)\sim_{\id_\AA\times \mathsf{a}_\PP} (y_1,z')\sim_{\mathsf{a}_\PP} z.\]
Taking the inner product of $(y_1,z')\sim_{\mathsf{a}_\PP} z$ with  $(0_h,0_m)\sim_{\ca{A}_\PP} w$, and using \eqref{eq:-1}, we find $\l w,z\r=0$.

The bundle of elements $(x,y)\in \AA\times \PP$ with $\sz(x)=\uz(y)$ has rank 
equal to $\on{rank}(\AA)+\on{rank}(\PP)-\dim\g=\on{rank}(\PP)+\dim\g$, which is also the rank of $\on{Gr}(\mathsf{a}_\PP)$. 
Hence  $x\circ y$ is defined if and only if $\sz(x)=\uz(y)$. The rest is clear.
\end{proof}

\begin{remark}\label{rem:weak}
\begin{enumerate}
\item In \cite{lib:dir}, Dirac Lie groups were defined in terms of a $\VB$-groupoid structure on $\AA$;  Theorem \ref{th:ingroupoid} shows that the Definition \ref{def:dir} is equivalent.
In turn, the description in \cite{lib:dir} is equivalent to the  super-geometric definition of \cite{lib:qua}, where it is shown that these are indeed the groups for the super-geometric incarnation of Dirac manifolds. 
\item Definition \ref{def:dir}  is similar to that in a 2007 preprint of  Milburn \cite{mil:two}, who defines `Dirac Lie groups' as group objects for various kinds of `Dirac categories'. However, his choice of morphisms for general Courant algebroids is much more restrictive then the one used here. 
\item
On the other hand,  our definition is \emph{different} from notions of `Dirac Lie group' in the work of Ortiz \cite{ort:mu} and  Jotz \cite{jot:dir} -- see Section \ref{subsec:jotz} below.  
\end{enumerate}
\end{remark}

In an analogous fashion, one can define `group-like' objects, and their actions, in the category $\LA^\vee$ of Lie algebroids and Lie algebroid comorphisms. One finds that if $(H,E)$ is an $\LA^\vee$ Lie group, then $E$ is a vacant $\VB$-groupoid 
$E\rra E^{(0)}$
with base $H\rra \pt$, and with $E^{(0)}=E_e=:\g$ as the units.  
Compatibility with the Lie algebroid structure means that $E$ is an \emph{$\LA$-groupoid}. Given an  $\LA^\vee$-action of $(H,E)$ on a Lie algebroid $(M,L)$, the Lie algebroid $L$ becomes an $\LA$-groupoid module. See Stefanini \cite[Chapter 3]{ste:mor} for related discussions. 
\begin{example}\label{ex:la1}
Any Lie group $H$ has an  $\LA^\vee$ Lie group structure $(H,E)$, defined by  
the action Lie algebroid $E=H\times \h$ for the infinitesimal conjugation action on $H$. The $\VB$-groupoid structure 
$E\rra \h$ is that of an action groupoid for the trivial $H$-action on $\h$. More generally, given a Lie algebra automorphism $\kappa$ of $\h$, one can consider the $\kappa$-twisted conjugation action; its action Lie algebroid again defines an $\LA^\vee$ Lie group structure. 
\end{example}
Given a Lie group $H$, one has the classification results 
\[ \begin{split}
\mbox{Poisson Lie group structures on $H$} &\longleftrightarrow \mbox{$H$-equivariant Manin triples \ \ \cite{dr:qu}}\\
\mbox{Dirac Lie group structures on $H$} &\longleftrightarrow \mbox{$H$-equivariant Dirac-Manin triples \ \ \cite{lib:dir}} \\
\mbox{$\LA^\vee$ Lie group structures on $H$} & \longleftrightarrow \mbox{$H$-equivariant Lie algebra triples\ \ \cite{mac:dou}}
\end{split}
\]
We will recall below how these correspondences come about, and generalize to actions over homogeneous spaces.

\subsection{Examples} We conclude this section with some examples of Dirac actions. 
\subsubsection{Dirac actions of Poisson Lie groups}\label{subsec:dirpl}
Let $H$ be a Lie group, and $\T H$ the standard Courant algebroid with the  multiplication morphism 
$\Mult_{\T H}=\T \Mult_H$. This morphism is the multiplication for the $\VB$-groupoid structure, 
given as the direct product of the group $TH\rra \pt$ with the cotangent groupoid $T^*H\rra \h^*$. 
As remarked in \cite{lib:dir}, a Dirac structure $(\T H,\,E)$ defines a Dirac Lie group structure 
for this multiplication morphism if and only if $E$ is the graph of a Poisson Lie group structure 
$\pi_H$. Hence, all Dirac Lie groups with $\AA=\T H$ are actually Poisson Lie groups (unless one relaxes the definitions (cf. Remark \ref{rem:weak}) or allows twistings by closed 3-forms, see below). 

Given an $H$-action on $M$, one obtains a $\VB$-groupoid action of $\T H$ on $\T M$, where $\mathsf{a}_{\T M}=\T \mathsf{a}_M$ is the product of the action $\mathsf{a}_{TM}=T\mathsf{a}_M$ of the group $TH$ on $TM$ with the `dual' action $\mathsf{a}_{T^*M}$ (cf.~ Appendix \ref{app:vb}) of the groupoid $T^*H\rra \h^*$ on $T^*M$; the moment map $\uz_{T^*M}\colon T^*M\to \h^*$ is the usual moment map  from symplectic geometry.
Given a Dirac structure $(\T M,\,L)$, this action defines a Dirac Lie group action of $(H,\,\T H,\,E)$ if and only if $L$ is a submodule over $E$. If $\pi_M$ is a Poisson structure and the $H$-action on $M$ is Poisson, then 
$L=\on{Gr}(\pi_M)$ is a submodule. But there are also many non-Poisson examples: For instance, 
$L=TM$ is always a submodule. 

\subsubsection{{Cartan-Dirac structure}}\label{ex:cartandirac}
Let $H$ be a Lie group with Lie algebra $\h$. Suppose $\h$ has an $\Ad(H)$-invariant metric, and denote by $\ol{\h}$ the same Lie algebra with the opposite metric. 
The metrized Lie algebra $\dd=\h\oplus \ol{\h}$ contains the diagonal 
$\g=\h_\Delta\subseteq \h\oplus \ol{\h}$ as a Lagrangian Lie subalgebra. 
The Lie algebra $\dd$ acts on $H$ by the difference of the left-invariant and right-invariant vector fields, 
\[ \varrho(\tau',\tau)=\tau^L-{\tau'}^R.\]
This action has coisotropic (in fact, Lagrangian) stabilizers. It defines an action Dirac structure  $(\AA,E)=(H\times \dd,\,H\times \g)$
on $H$, known as the \emph{Cartan-Dirac structure}.\footnote{One can identify 
$\AA$ with the twist of the standard Courant algebroid $\T H$ by the Cartan 3-form $\eta\in \Om^3(H)$ defined by the metric (see \cite[Example 4.11]{bur:di} 
or \cite[Section 3.1]{al:pur}). This reproduces the perhaps more familiar description of the Cartan-Dirac structure.}  The Courant algebroid is a $\VB$-groupoid $\AA\rra \g$, 
given as the direct product of  the group $H$ with the pair groupoid of $\h$:
Thus $\sz(h,\tau',\tau)=\tau,\ \ \tz(h,\tau',\tau)=\tau'$, and for composable elements
\[ (h_1,\tau'_1,\tau_1)\circ (h_2,\tau_2',\tau_2)=
(h_1h_2,\tau_1',\tau_2).\]
As shown in \cite[Section 3.4]{al:pur}, the groupoid multiplication $\Mult_\AA$ is a Dirac morphism, hence 
$(H,\AA,E)$ is a Dirac Lie group. 
Note that the orbits of the Lie algebroid $E=H\times \g$ are the conjugacy classes in $H$. As explained in  \cite{bur:di, bur:cou}, this Dirac Lie group structure is responsible for the theory of quasi-Hamiltonian spaces \cite{al:mom}.
Suppose $M$ is a manifold with an action of $\dd$, with coisotropic stabilizers, and let 
$\PP=M\times \dd$ be the action Courant algebroid. Suppose that the action of the sub-Lie algebra  
$\h$ integrates to an action $\mathsf{a}_M$ of $H$. Then the groupoid action $\mathsf{a}_\PP$ of $\AA\rra \g$, given by the product of the action $\mathsf{a}_M$ with the groupoid multiplication $\Mult_\dd$,  is a Courant morphism. 
Any Lagrangian Lie subalgebra $\mathfrak{l}\subseteq  \dd$ defines a Dirac structure $L=M\times \mathfrak{l}$ in $\PP$, which is a submodule for the groupoid action of $E$. That is, we have a 
Dirac action of $(H,\AA,E)$ on $(M,\PP,L)$. 

An interesting special case is the following: Consider any coisotropic Lie subalgebra $\mf{c}\subseteq  \h\times\ol{\h}$,  and let $M=D.\cc\subseteq \on{Grass}_r(\dd)$ be its orbit under the adjoint action of  $D=H\times H$ on the Grassmannian of $r=\dim\cc$-dimensional subspaces. The infinitesimal action of $\dd$  has coisotropic stabilizers, since the stabilizer at $\mf{c}$ contains $\mf{c}$ itself. 

\subsubsection{Wonderful compactification}\label{subsubsec:wonderful}
We continue the discussion from \ref{ex:cartandirac}, but with the assumption that $H$ is a connected complex semi-simple Lie group with trivial center. In this case there is an embedding 
$H\to \on{Grass}_{\dim H}(\dd)$, taking $h\in H$ to the graph of the adjoint action, 
$\on{Gr}(\Ad_h)$. Equivalently, $H$ is embedded as the $D$-orbit of the diagonal. According to deConcini-Procesi,  the closure of $H$ inside the Grassmannian is a smooth submanifold $M$. It is called the \emph{wonderful compactification}. 
Since the infinitesimal action of $\dd$ on $H$ has co-isotropic stabilizers, the same is true for the action on its closure $M$. Thus, by the above, any choice of a Lagrangian Lie subalgebra $\mathfrak{l}$ (in particular $\mathfrak{l}=\g$)
makes $M$ into a Dirac manifold $(M,\PP,L)$, with a Dirac action of $(H,\AA,E)$. 
  
For $\h$ semisimple, there is a classification of Lagrangian Lie subalgebras of $\h\oplus \ol{\h}$, due to Karolinsky \cite{kar:cla} and Delorme \cite{del:cla}.   A detailed discussion of the variety of Lagrangian Lie subalgebras of $\dd=\h\oplus \ol{\h}$, and its relation with Poisson homogeneous spaces, 
may be found in the work of Evens and Lu \cite{ev:on,ev:on2}. 

\subsubsection{Quasi-Poisson actions}\label{subsubsec:sv}
In a  recent paper, \v{S}evera-Valach \cite{sev:lie} develop a theory of \emph{quasi-Poisson Lie groups} $H$. These are classified by $H$-equivariant \emph{Manin quadruples} $(\dd,\h,\g,\h')_\beta$, consisting of a Lie algebra $\dd$ with invariant metric $\beta$, 
and with three Lie subalgebra such that $\dd=\h\oplus \g\oplus \h'$. Here the metric restricts to a metric on $\g$ and to zero on both $\h,\h'$;  the adjoint action of $\g$ is required to preserves both $\h$ and $\h'$. An example is the usual triangular decomposition of a semisimple complex Lie algebra, and more general decompositions involving parabolics. (See \cite{sev:lie}.)  Without getting into details, we remark that the  quasi-Poisson Lie groups can be studied from  the perspective of Dirac Lie groups; the relevant Dirac Manin triple is $(\dd,\g\oplus \h',\h)_\beta$. Similarly, the q-Poisson actions can be regarded as Dirac Lie group actions.

\subsection{Comparison with the Ortiz-Jotz theory}\label{subsec:jotz}
As already remarked, our approach to Dirac Lie groups  and their actions is different from that in the work of Ortiz \cite{ort:mu} and  Jotz \cite{jot:dir}. 

In Ortiz' paper \cite{ort:mu}, a multiplicative Dirac structure on a Lie group $H$ is a Dirac structure $E\subset \T H$ inside the standard Courant algebroid, such that $E$ is a subgroupoid of $\T H\rra \h^*$, but not necessarily a wide subgroupoid: the space of units of $E$ may be strictly smaller than $\h^*$. He refers to a Lie group with multiplicative Dirac structure as a Dirac Lie group, and shows that, modulo a `regularity assumption', all of these are obtained by pull-back of a Poisson Lie group structure under a surjective group homomorphism. Jotz \cite{jot:dir} remarks that Ortiz's definition is equivalent to group multiplication of $H$ being weakly Dirac, and obtains a more precise version of the classification theorem. (Here,  by \emph{weakly Dirac} we mean
that one omits the uniqueness part in condition ($\ca{D}$) from \ref{subsec:dircat} -- 
in many references, including \cite{jot:dir}, these are referred to as Dirac maps.) In a similar fashion, Jotz studies Dirac actions on a homogeneous space $M=H/K$, equipped with Dirac structure  $L\subset \T M$, by requiring that the action map is weakly Dirac, and again she obtains a classification of such actions. 

For the more general Dirac structures $(\AA,E)$ in this paper,  in order to demand that `group multiplication is a Dirac morphism', it is necessary to specify a multiplication morphism  $\Mult_\AA$.  We work with the \emph{strong} notion of Dirac morphism in order to have an actual category. As mentioned before, if $\AA=\T H$ this only allows Poisson Lie group structures, whereas the Ortiz-Jotz definition also includes, e.g., $E=TH$.  On the other hand, the examples \ref{ex:cartandirac}, \ref{subsubsec:wonderful} and \ref{subsubsec:sv} are not included in their theory. 

One could relax our definitions in two directions. similar to Ortiz-Jotz: either (i) allow weak Dirac morphisms (keeping in mind that these cannot always be composed), or (ii)  study multiplicative Dirac structures over $H$ (given by a $\ca{CA}$-groupoid structure $\AA\rra \g$ such that the Dirac structure $E\subset \AA$ is a subgroupoid), and modules over these. These two directions are different: e.g., having $\Mult_\AA$ only weakly Dirac 
does not determine a groupoid structure on $\AA$. For both generalizations,  no general classification results are in sight.

\section{$\LA^\vee$-actions}
Recall again that forgetting the ambient Courant algebroid is a functor $\ca{DIR}\to \LA^\vee$. In particular, any Dirac Lie group structure on $H$ defines an $\LA^\vee$-Lie group structure. We will hence begin by considering $\LA^\vee$-actions.  Aside from Poisson Lie groups with the cotangent Lie algebroid 
$E=T^*_\pi H$, the main example to keep in mind is the $\LA^\vee$-Lie group structure $(H,H\times \h)$ from Example \ref{ex:la1}. 

\subsection{General properties of $\LA^\vee$ actions}\label{sec:general}
Let $(H,E)$ be an $\LA^\vee$-Lie group. The unit fiber $\g=E_e$ inherits a Lie algebra structure such that 
the inclusion $\g\subseteq E$ is a  sub-Lie algebroid along $\{e\}\subseteq H$. As we observed above, the $\VB$-groupoid $E\rra \g$ is \emph{vacant}: Its source and target maps are fiberwise isomorphisms. We will use the source map to define a trivialization: 
\[ E=H\times \g.\] 
Let us temporarily forget about the Lie algebroid structure and just consider any vacant $\VB$-groupoid 
$E\rra \g$ over $H\rra \pt$.  Let $V$ be a $\VB$-module, with base $M$ and moment map $\uz_V\colon V\to \g$. We obtain an $H$-action by bundle automorphisms of $V$, covering the 
action $\mathsf{a}_M\colon H\times M\to M$ on the base, given by 
\begin{equation}\label{eq:acc} h\bullet y:=x\circ y,\end{equation}
for $y\in V$ and $h\in H$; here $x\in E$ is the unique element in the fiber $E_h$ for which $\sz_E(x)=\uz_V(y)$. 

For a first application, recall that any groupoid acts on its space of units; for a $\VB$-groupoid this makes the bundle of units into a 
$\VB$-module.  Hence $V=\g$, as a vector bundle over $M=\pt$, is a $\VB$-module over $E\rra \g$, and we obtain an action $\bullet$ of $H$ on $\g$. By definition, this action is such that 
\begin{equation} \label{eq:ths} \tz_E(x)=h\bullet \sz_E(x)\end{equation}
for all $x\in E_h$. For a general $\VB$-module $V$ over $E\rra \g$, the moment map $\uz_V\colon V\to \g$ is a morphism of $\VB$-modules; 
hence it has the equivariance property  
\begin{equation} \label{eq:momeq}
\uz_V(h\bullet y)=h\bullet \uz_V(y).\end{equation}
(Alternatively, this follows from \eqref{eq:acc} and \eqref{eq:ths}, since $\uz_V(h\bullet y)=\tz_E(x)$ and $\uz_V(y)=\sz_E(x)$.) 

By a result of Ping Xu \cite{xu:poigr}, for any Poisson action of a Poisson Lie group $H$ 
on a Poisson manifold $M$, the (symplectic) moment map $T^*_\pi M\to \h^*=\g$ is a morphism of Lie algebroids. The following proposition is a similar result for any  
$\LA^\vee$-action of $(H,E)$ on $(M,L)$. 
%The base action of $H$ on $M$ will be denoted $\mathsf{a}_M$. 
%
\begin{proposition}\label{prop:uzl}
The moment map $\uz_L\colon L\to \g$ is a morphism of Lie algebroids: That is,
\[ \uz_L([\sigma_1,\sigma_2])=[\uz_L(\sigma_1),\uz_L(\sigma_2)]_\g+\L_{\mathsf{a}_L(\sigma_1)}\uz_L(\sigma_2)-
\L_{\mathsf{a}_L(\sigma_2)}\uz_L(\sigma_1)\]
for all sections $\sigma_1,\sigma_2\in \Gamma(L)$. 
In particular, both $\sz_E,\tz_E\colon E\to \g$ are morphisms of Lie algebroids. 
\end{proposition}
\begin{proof} 
The Lie algebroid comorphism $\mathsf{a}_L\colon E\times L\da L$ restricts to a Lie algebroid comorphism 
$\g\times L\da L$, with base map the identity map of $M$. Under this comorphism, $(\uz_L(\sigma),\sigma)\sim\sigma$ for all 
sections $\sigma$. Given two sections $\sigma_i\in\Gamma(L)$, by taking Lie brackets of 
the relations 
$(\uz_L(\sigma_i),\sigma_i)\sim\sigma_i,\ i=1,2$, we obtain 
\[  \Big([\uz_L(\sigma_1),\uz_L(\sigma_2)]_\g+\L_{\mathsf{a}_L(\sigma_1)}\uz_L(\sigma_2)-
\L_{\mathsf{a}_L(\sigma_2)}\uz_L(\sigma_1),\,[\sigma_1,\sigma_2]\Big)
\sim
[\sigma_1,\sigma_2].\]
By the uniqueness part in the definition of $\LA^\vee$-morphism, the first entry on the 
left hand side must be $\uz_L([\sigma_1,\sigma_2])$. 
\end{proof}
In particular, we see that the space of sections with $\uz_L(\sigma)=0$   is a Lie subalgebra of $\Gamma(L)$, and those annihilated by both $\uz_L$ and $\mathsf{a}_L$ form a Lie ideal. 
\begin{proposition}\label{prop:l-equivariant}
The space $\Gamma(L)^H$ of $\bullet$-invariant sections is closed under the 
Lie algebroid bracket. Furthermore, the  
difference $\mathsf{a}_L(h\bullet y)-h.\mathsf{a}_L(y)$ is tangent to $H$-orbits, for all $h\in H$ and $y\in L$. Hence, if the $H$-action on $M$ is a principal action, then $L/H$ is a Lie algebroid over 
$M/H$. 
\end{proposition}
\begin{proof}
Since $\mathsf{a}_L$ is a Lie algebroid comorphism, the pull-back map $\mathsf{a}_L^*\colon \Gamma(L)\to \Gamma(E\times L)$ 
preserves Lie brackets. Let $S\colon E\times L\to L$ be the Lie algebroid morphism given by projection to the second factor. Then $\sigma$ is invariant if and only if $\mathsf{a}_L^*\sigma\sim_S\sigma$. Given two invariant sections $\sigma_1,\sigma_2$, we conclude 
\[ \mathsf{a}_L^*[\sigma_1,\sigma_2]=[\mathsf{a}_L^*\sigma_1,\mathsf{a}_L^*\sigma_2]\sim_S [\sigma_1,\sigma_2],\]
hence $[\sigma_1,\sigma_2]$ is invariant. 

If $x\in E_h$ and $y\in L_m$ with $\sz_E(x)=\uz_L(y)$, then $\mathsf{a}_L(h\bullet y)=\mathsf{a}_L(x\circ y)=\mathsf{a}_E(x)\circ \mathsf{a}_L(y)$, with the action of $TH$ on $TM$. But for any $v\in T_hH,\ w\in T_mM$ the difference
$v\circ w-h.w$ is tangent to orbits. 
The rest is clear. 
\end{proof} 

\subsection{The Lu-Lie algebroid}
 In her 1997 paper \cite{lu:poi}, Jiang-Hua Lu proved that for any Poisson action of a Poisson Lie group $H$ on a Poisson manifold $(M,\pi_M)$, the action Lie algebroid $M\times \h$ and the cotangent Lie algebroid 
$T^*_\pi M\cong \on{Gr}(\pi_M)$ form a \emph{matched pair}, in the sense of Mokri \cite{mok:mat}. 
Equivalently, their direct sum 
\begin{equation} \label{eq:lulie}
 \wh{T^*_\pi M}:=(M\times\h)\oplus T^*_\pi M\end{equation}
has a Lie algebroid structure, in such a way that 
the two summands are Lie subalgebroids. Lu proved furthermore that the action of $H$ on $M$ lifts to an 
action by Lie algebroid automorphisms on \eqref{eq:lulie}, and that the direct sum of the projection $M\times \h\to \h$ and  the moment map 
$T^*_\pi M\to \h^*\cong \g$ defines a Lie algebroid morphism to $\dd=\h\oplus \g$, the Drinfeld double of $\h$. 
Further aspects of this Lie algebroid, relating it to a construction of Bursztyn and Crainic 
\cite{bur:di} for quasi-Poisson actions, are discussed in the article \cite{kos:poi} of Kosmann-Schwarzbach.
We have the following result for general $\LA^\vee$-actions
of $(H,E)$ on Lie algebroids $(M,L)$. 

\begin{theorem}\label{th:lulie}
The direct sum 
\begin{equation}\label{eq:leqn}
 \wh{L}=(M\times \h)\oplus L\end{equation}
 has the structure of an $H$-equivariant Lie algebroid, 
in such a way that both summands are sub-Lie algebroids, and such that the
map $\h\to \Gamma(\wh{L})$ given by the constant sections of $M\times \h$ gives generators for the $H$-action. 
\end{theorem}
\begin{proof}
The $\LA^\vee$-action of $(H,E)$ on $(M,L)$ extends to an action on $(H\times M,\,
TH\times L)$, by putting $\uz_{TH\times L}(v,y)=\uz_L(y)$, and 
for composable elements 
\[ x\circ (v,y)=(\mathsf{a}_E(x)\circ v,\,x\circ y),\]
where $\mathsf{a}_E(x)\circ v$ is a product in the group $TH$. To see 
its compatibility with Lie brackets, note that  $\on{Gr}(\mathsf{a}_{TH\times L})$ can be regarded as the intersection of $\on{Gr}(\Mult_{TH})\times \on{Gr}(\mathsf{a}_L)$ 
with the set of all $((v',w,v),\ (y',x,y))$ such that $w=\mathsf{a}(x)$. Being the intersection of two 
Lie subalgebroids, it is itself a Lie subalgebroid. 

The action of $E\rra \g$ on $TH\times L$ gives rise to a $\bullet$-action of $H$, which commutes with the action $T\mathsf{a}^R(h)\times \id_L$. Hence, Proposition \ref{prop:l-equivariant}
shows that the quotient by the $\bullet$--action is a Lie algebroid $\wh{L}=(TH\times L)/H$ with an action of $H$ by automorphisms. As a vector bundle, 
\[ \wh{L}\cong (TH\times L)\big|_{\{e\}\times M}=(M\times \h)\oplus L.\] 
The first summand is the quotient of the Lie subalgebroid $TH\times M\subseteq TH\times L$ under the $\bullet$-action, hence it is itself a Lie subalgebroid. Since  left-trivialization identifies $TH\times M$ with the action Lie algebroid for $\mathsf{a}^R$, it follows that $M\times \h$ 
is the action  Lie algebroid $M\times \h$ for $\mathsf{a}_M$.  

It remains to show that $L$ is a Lie subalgebroid of $\wh{L}$. 
For $\sigma\in \Gamma(L)$, let $\breve{\sigma}\in 
\Gamma(TH\times L)^H$ be the corresponding invariant section, so that 
$\breve{\sigma}|_{\{e\}\times M}=\sigma$. We need to show that the  map $\sigma\mapsto \breve{\sigma}$ preserves brackets. This map can be written as a composition
\[ \Gamma(L)\xra{(\on{Inv}_E\times \id)\circ \mathsf{a}_L^*}     \Gamma(E\times L)^H 
 \xra{\mathsf{a}_E\times \id} \Gamma(TH\times L)^H;\]
here $\mathsf{a}_L^*\colon \Gamma(L)\to \Gamma(E\times L)$ is the pull-back map under the 
Lie algebroid comorphism  $\mathsf{a}_L\colon E\times L\da L$,  and $\on{Inv}_E\colon E\to E$ is the inversion. 
(In detail: If $\sig_m=y$, then $(\on{Inv}_E\times \id)((\mathsf{a}_L^*\sig)_{h,m})=(x,\,x\circ y)$, with the unique $x\in E_h$ such that $\sz(x)=\uz(y)$. Hence $\breve{\sig}_{h,m}=(\mathsf{a}_E(x),\,x\circ y)$. If $h=e$, then $x\in\g=E_h$, hence $\mathsf{a}_E(x)=0_e$. That is, $\breve{\sig}_{e,m}=\sigma_m$.) 
Since all of these maps preserve brackets, 
it follows that the map $\sigma\mapsto \breve{\sigma}$ is a Lie algebra morphism, as required. 
\end{proof}
Clearly, the construction from Theorem \ref{th:lulie} is functorial: Given $\LA^\vee$-actions of $(H,E)$ on $(M_1,L_1),\ (M_2,L_2)$, and a Lie algebroid morphism $L_1\to L_2$
intertwining the actions, the resulting map $\wh{L}_1\to \wh{L}_2$
is an $H$-equivariant Lie algebroid morphism. 

As a first example, we can apply Theorem \ref{th:lulie} to the action of $(H,E)$ on $(\pt,\g)$. We find that 
\begin{equation}\label{eq:dd}
\wh{\g}=\h\oplus \g
\end{equation} 
is a Lie algebra, with $H$ acting by automorphisms, and with $\h$ and $\g$ as Lie subalgebras. Writing $\dd:=\wh{\g}$, 
this is the $H$-equivariant Lie algebra triple $(\dd,\g,\h)$ associated to the $\LA^\vee$-Lie group $(H,E)$
by Mackenzie's classification \cite{mac:dou}.
\begin{proposition}\label{prop:fl}
The map
\[ f_{\wh{L}}\colon \wh{L}=\h\times L\to \dd,\ \ (\tau,y)\mapsto \tau+\uz_L(y)\]
is an $H$-equivariant morphism of Lie algebroids.
\end{proposition}
\begin{proof}
The moment map $\uz_L\colon L\to \g$ is a morphism of Lie algebroids, and is 
equivariant for the action of $(H,E)$.  Hence the result follows by functoriality. 
\end{proof}
We will denote by $\pr_\h,\,\pr_\g$ the projections from $\dd=\h\oplus \g$ to the two summands. 
In terms of $f_{\wh{L}}$, the $H$-action on  $\wh{L}=\h\times L$ is explicitly given as
\[ h.(\tau,y)=\Big(\pr_\h \big(\Ad_h f_{\wh{L}}(\tau,y)\big),\ h\bullet y\Big).\]

%We will also need the following  fact. Let $\pr_L\colon \wh{L}\to L$ denote the projection along $M\times \h$. Recall that $\Gamma(L)^H$ is a sub-Lie algebra of $\Gamma(L)$.
%\begin{lemma}\label{lem:project}The map $\pr_L\colon \Gamma(\wh{L})^H\to \Gamma(L)^H$ is a Lie algebra homomorphism.\end{lemma}\begin{proof}Let $\sigma_1,\sigma_2\in \Gamma(L)^H$ and $\sigma_1',\sigma_2'\in \Gamma(\wh{L})^H$, with $\pr_L\sig_i'=\sig_i$. Note that 
%\[ [\Gamma(M\times \h),\Gamma(\wh{L})^H]\subseteq \Gamma(M\times \h),\]
%since the constant sections of $M\times \h$ are generators for the $H$-action on $\Gamma(\wh{L})$. Hence, with $\phi_i=\sig_i'-\sig_i$, \[ [\sig_1,\sig_2]=[\sig_1'-\phi_1,\sig_2'-\phi_2]=[\sig_1',\sig_2']-[\phi_1,\sig_2']+[\phi_2,\sig_1']+[\phi_1,\phi_2].\]The last three terms are in $\Gamma(M\times \h)$, proving that  $\pr_L [\sig_1',\sig_2']=[\sig_1,\sig_2]$. \end{proof}

\subsection{The action of $\dd$ on $H$}
There are two commuting $\LA^\vee$-actions of $(H,E)$ on itself, with base actions the
actions $\mathsf{a}^L,\mathsf{a}^R$ of $H$ on itself: 
\[ (x,y)\sim_{\mathsf{a}^L_E} x\circ y,\ \ \ (x,y)\sim_{\mathsf{a}^R_E} y\circ x^{-1}\] 
Apply Theorem \ref{th:lulie} to the action  $\mathsf{a}^R_E$. We obtain an $H$-equivariant Lie algebroid
\[ \wh{E}=(H\times \h)\oplus E\]
together with an $H$-equivariant Lie algebroid morphism $f_{\wh{E}}\colon \wh{E}\to \dd$; 
here $H\times \h$ is the action Lie algebroid for $\mathsf{a}^R$. 
Since $\mathsf{a}^R_E$ commutes  the action $\mathsf{a}^L_E$, 
the latter defines an $\LA^\vee$-action of $(H,E)$ on $(H,\wh{E})$, commuting with the 
action of $H$ by Lie algebroid automorphisms, and such that $f_{\wh{E}}(x\circ z)=f_{\wh{E}}(x)$ for composable elements $x\in E$ and $z\in \wh{E}$. 
Since $f_{\wh{E}}$ is  a fiberwise isomorphism, it defines a trivialization. 
\begin{proposition}\label{prop:ehate}
The trivialization $\wh{E}=H\times \dd$  defined by  $f_{\wh{E}}$ identifies $\wh{E}$ 
with the action Lie algebroid for the dressing action \eqref{eq:action}.  The $H$-action by automorphisms 
reads as $g.(h,\lambda)=(hg^{-1},\Ad_g(\lambda))$, and 
%the anchor map is given by the dressing action of $\dd$ on $H$, i.e. in left trivialization 
%$TH=H\times \h$ 
%\begin{equation} \label{eq:actnow}
%\mathsf{a}_{\wh{E}}(h,\lambda)=\Ad_{h^{-1}}\pr_\h \Ad_h(\lambda).
%\end{equation}
%
the groupoid action of $E=H\times \g\rra \g$ on $\wh{E}$ is given by the moment map 
\begin{equation}\label{eq:uznow} 
\uz_{\wh{E}}(h,\lambda)=\pr_\g \Ad_h(\lambda),\end{equation}
and $(g,\xi)\circ (h,\lambda)=(gh,\lambda)$
for $\xi=\uz_{\wh{E}}(h,\lambda)$. 
\end{proposition}
\begin{proof}
The description of the $H$-action follows by $H$-equivariance of $f_{\wh{E}}$. 
Using the $H$-equivariance of  the 
anchor $\mathsf{a}_{\wh{E}}$, we find 
\[\mathsf{a}_{\wh{E}}(h,\lambda)=\mathsf{a}_{\wh{E}}\big(h^{-1}.(e,\Ad_h(\lambda))\big)=
T\mathsf{a}^R(h^{-1})\mathsf{a}_{\wh{E}}(e,\Ad_h(\lambda))\big)
=T\mathsf{a}^R(h^{-1})\pr_\h \Ad_h(\lambda)\]
proving $\mathsf{a}_{\wh{E}}(h,\lambda)=\varrho(\lambda)_h$. (See \eqref{eq:action}.) 
Similarly, $H$-invariance of $\uz_{\wh{E}}$ together with $\uz_{\wh{E}}(e,\lambda)=\pr_\g(\lambda)$
implies
\eqref{eq:uznow}. The formula for the groupoid action of $(H,E)$ follows since 
$f_{\wh{E}}$ is invariant under this action. 
\end{proof}
The trivialization $\wh{E}=H\times \dd$ restricts to the trivialization 
$E=H\times \g$ given by $\sz_E$. We conclude that $E$ is  isomorphic to 
the action Lie algebroid  for the dressing action of $\g\subseteq\dd$ on $H$. On the other hand, as a $\VB$-groupoid it is the action groupoid $E\rra \g$ for the action $h\bullet\xi=\pr_\g(\Ad_h\xi)$. 
This shows that the $\LA^\vee$ Lie group $(H,E)$ is fully determined by the $H$-equivariant Lie algebra triple $(\dd,\g,\h)$. To complete a proof of Mackenzie's classification, one has to show that conversely, given an $H$-equivariant triple $(\dd,\g,\h)$, the Lie algebroid and $\VB$-groupoid structure on 
$E=H\times\g$ are compatible, in the sense that 
the graph of the groupoid multiplication is a sub-Lie algebroid. (This may be verified directly, in the trivialization.) 

We have the following alternative description of $\wh{E}$:
\begin{proposition}
The map $(\mathsf{a}_{\wh{E}},\uz_{\wh{E}})\colon \wh{E}\to TH\times \g$ is 
an isomorphism of Lie algebroids. In terms of this identification, the $\VB$-groupoid
action of $E\rra \g$ is given by 
\[ \uz_{\wh{E}}(v,\xi)=\xi,\ \ x\circ (v,\xi)=(\mathsf{a}(x)\circ v,\ h\bullet \xi).\]
\end{proposition}
\begin{proof}
The map $(\mathsf{a}_{\wh{E}},\uz_{\wh{E}})$
is a Lie algebroid morphism, since both components are. In terms of left trivialization $TH=H\times \h$ and the 
trivialization $\wh{E}=H\times\dd$, the map is given at the fiber of $h\in H$ by 
$\zeta\mapsto (\Ad_{h^{-1}}\pr_\h (\Ad_h \zeta), \pr_\g (\Ad_h \zeta))$. 
It is an isomorphism with inverse
\begin{equation}\label{eq:THg}
 T_hH\times \g\to \wh{E}_h,\ \  (\tau,\xi)\mapsto \tau+\Ad_{h^{-1}}\xi.
\end{equation}
The rest is clear. 
\end{proof}
Note that the groupoid action of $E\rra\g$ is  just the action used in the construction of $\wh{\g}=\dd$. We recover that the quotient $\wh{E}/H$ by the $\bullet$-action is $\dd$. 

\begin{remark}
Let $f_{\wh{L}}\colon \wh{L}\to \dd$ be the Lie algebroid morphism from 
Proposition \ref{prop:fl}, and let $p\colon H\times M\to M$ be the projection to the second factor.  Using the dressing action \eqref{eq:action} of $\dd$ on $H$, the vector bundle pullback  $p^*\wh{L}$ 
has a Lie algebroid structure. It is not hard to see that $p^*\wh{L}\cong TH\times L$, 
the Lie algebroid used in the construction of $\wh{L}$. 
\end{remark}
%
% \begin{remark}The  $\VB$-groupoid structure on  $\wh{E}=H\times\dd\rra \dd$ as an action groupoid for the $H$-action on $\dd$ is \emph{not} compatible with the Lie algebroid structure (the graph of the groupoid multiplication is \emph{not} a sub-Lie algebroid) unless $H$ is discrete. Indeed, compatibility would imply that the anchor is the zero morphism on the space of units.\end{remark}

\subsection{Homogeneous spaces}\label{subsec:homla}
We now specialize to the case that $M=H/K$ a homogeneous space.  The kernel of the anchor of the action Lie algebroid $M\times\h$ is an associated bundle $H\times_K\k$, where $\k$ is the Lie algebra of $K$. Since $\wh{L}$ contains the action Lie algebroid, it is a \emph{transitive} Lie algebroid: its anchor is surjective. This defines an $H$-equivariant exact sequence of Lie algebroids 
\[ 0\to \ker(\mathsf{a}_{\wh{L}})\to \wh{L}\to TM\to 0.\]
Since $\ker(\mathsf{a}_{\wh{L}})$ has a trivial anchor, it is an $H$-equivariant bundle of Lie algebras
\[ \ker(\mathsf{a}_{\wh{L}})=H\times_K \uu,\]
where $\uu=\ker(\mathsf{a}_{\wh{L}})_e$ is a Lie algebra with an action of $K$ by automorphisms. 
The inclusion of $\ker(\mathsf{a}_{M\times \h})\cong H\times_K\k$ into 
$\ker(\mathsf{a}_{\wh{L}})$ restricts to a $K$-equivariant Lie algebra morphism $\k\to \uu$, thus $(\uu,K)$ is a 
Harish-Chandra pair (cf.~ Section \ref{subsec:hc}). The  $H$-equivariant Lie algebroid morphism
$f_{\wh{L}}\colon \wh{L}\to\dd$ restricts to a $K$-equivariant Lie algebra morphism 
\[ f_{\uu}\colon \uu\to \dd,\]
whose restriction to $\k\subseteq \uu$ coincides with $T_e\phi\colon \k\to \h$, 
where  $\phi\colon K\hra H$ is the inclusion. It hence defines a morphism of 
Harish-Chandra pairs from $(\uu,K)$ to $(\dd,H)$. 
These data give a full classification: 
\begin{proposition}\label{prop:homla}
The Lie algebroids $(H/K,L)$ together with 
$\LA^\vee$-actions of $(H,E)$ are classified by 
Harish-Chandra pairs $(\uu,K)$ along with a morphism of Harish-Chandra pairs 
\begin{equation}\label{eq:hcmor} (f_\uu,\phi)\colon (\uu,K)\to (\dd,H).\end{equation}
As a Lie algebroid, $L$ is the reduction by $K$ of the action Lie algebroid for $\varrho\circ f_\uu\colon \uu\to \Gamma(TH)$:
\[ L=(H\times \uu)\qu K=H\times_K (\uu/\k).\]
The $\VB$-groupoid action of $E\rra \g$ on $L$ is given by the moment map 
\[ \uz_L([(h,\zeta \mod\k)])=\pr_\g(\Ad_h(f_\uu(\zeta)),\]
and for composable elements, $(g,\xi)\circ  [(h,\zeta\mod \k)]=[(gh,\zeta\mod\k)]$.
\end{proposition}
\begin{proof}
(i) We begin with the case $K=\{e\}$, so that $M=H$.  Suppose $(M,L)$ and the action are given. Since $K=\{e\}$, the anchor map restricts to an isomorphism on $H\times \h$, 
which hence has $\ker(\mathsf{a}_{\wh{L}})$ as an $H$-invariant complement. The projection $\pr_L\colon \wh{L}\to L$ along $H\times \h$ restricts to an $H$-equivariant vector bundle isomorphism 
\[ \pr_L\colon H\times \uu=\ker(\mathsf{a}_{\wh{L}})\xra{\cong} L.\] 
By Proposition \ref{prop:l-equivariant}, this map is a Lie algebra morphism on $H$-invariant sections. This shows that the bracket on invariant sections of $L$ is given by the Lie bracket on $\uu$; hence $L$ is an action Lie algebroid for some $\uu$-action on $H$. 
To compute the anchor and the moment map in terms of this the trivialization,  let $y\in L_h$ be given. Let $y=y'+y''$ 
be its decomposition in $\wh{L}_h$, into 
\[ y'=(h,\zeta)\in \ker(\mathsf{a}_{\wh{L}})=H\times \uu,\ \ \ \ 
y''=(h,\tau)\in H\times \h.\] 
We obtain 
\[ \uz_L(y)=
f_{\wh{L}}(y)=
f_{\wh{L}}(y')+f_{\wh{L}}(y'')=\Ad_h(f_\uu(\zeta))+\tau.\] 
The $\g$-component of this equation shows 
\begin{equation}\label{eq:lmo} \uz_L(y)=\pr_\g(\Ad_h(f_\uu(\zeta)),\end{equation}
while the $\h$-component tells us that $\tau=-\pr_\h (\Ad_h(f_\uu(\zeta)))$. Since $H\times \h$ 
is identified with the action Lie algebroid for the action $\mathsf{a}^L_H$, its anchor map is $(h,\tau)\mapsto
-\tau^R|_h=-\Ad_{h^{-1}}\tau$, using  left trivialization $TH=H\times \h$. 
Since $\mathsf{a}_L(y)=\mathsf{a}_{\wh{L}}(y'+y'')=\mathsf{a}_{\wh{L}}(y'')=-\tau^R|_h$, this shows 
\begin{equation}\label{eq:lact} 
\mathsf{a}_L(y)=\Ad_{h^{-1}}\pr_\h \Ad_h f_\uu(\zeta)
=\varrho(f_\uu(\zeta))_h
.\end{equation}
%
%(ii) Since $L\cong H\times \uu$ is an action Lie algebroid, the map $\pr_\uu\colon L\to \uu$ defined by this trivialization is a morphism of Lie algebroids.  Consider the embedding 
%\[ L\to \wh{E}\times \uu=(TH\times \g)\times \uu,\ \ y\mapsto (\mathsf{a}_L(y),\ \uz_L(y),\ \kappa(y)).\] 
%It is a Lie algebroid morphism, since all three components are, and it intertwines the actions of $(H,E)$ on $L$ and on $\wh{E}\times \uu$. The formula from (i) show that the of $L$ under this map consists of all $(z,\zeta)\in \wh{E}\times \uu$ such that $f_{\wh{E}}(z)=f_\uu(\zeta)$. 

(ii) For general $K$, let $p\colon H\to H/K$ be the projection. The $\LA^\vee$-action of 
$(H,E)$ on $L$ lifts to an action on $p^!L\subseteq TH\times L$. Hence, the discussion above applies, and identifies $p^!L\cong H\times \uu$. Note that this identification 
intertwines the $K$-action on $p^!L$ with the action $k.(h,\zeta)=(hk^{-1},\,\Ad_k\zeta)$, 
and intertwines the generators for these actions. Hence 
$L=p^!L\qu K=(H\times \uu)\qu K$. The moment map $\uz_{p^!L}$, being $K$-invariant, descends to the quotient, and so does the action of $(H,E)$.

%produces an embedding \[ p^!L\hra \wh{E}\times \uu\] as the Lie subalgebroid consisting of all $(z,\zeta)$ such that $f_{\wh{E}}(z)=f_\uu(\zeta)$. This embedding is $K$-equivariant, using the given action of $K$ on $\uu$ and the action on $\wh{E}$ as a subgroup of $H$. Furthermore, it also intertwines the generators, given by the generators $\k\to \Gamma(\wh{E})$ (as constant sections of $H\times \dd$) and $\k\to \uu$ (the inclusion described above). Hence $L=(p^!L)\qu K$ is embedded as a Lagrangian subbundle of $(\wh{E}\times \uu)\qu K$. 

(iii) The discussion above shows that $(H/K,L)$ is uniquely determined by 
$(\uu,K)$ and the morphism \eqref{eq:hcmor}. Conversely, given these data, 
define $L=(H\times \uu)\qu K$, with the moment map $\uz_L$, and the $(H,E)$-action as above. We need to show that the action $\mathsf{a}_L$ is an $\LA$-comorphism.  
Rather than proving this directly, consider the embedding 
\[ p^!L=H\times \uu\to \wh{E}\times \uu,\ \ (h,\zeta)\mapsto ((h,f_\uu(\zeta)),\zeta).\]
This is a $K$-equivariant morphism of Lie algebroids, and intertwines the moment maps, as well as the actions of $E\rra \g$. Since the $(H,E)$-action on $\wh{E}\times \uu$ 
(given by the action on the first factor) is an 
$\LA$-comorphism, the same is true for the action on the Lie subalgebroid $p^!L$, and hence on $L$.  
%we obtain $(H/K,L)$ with the action of $(H,E)$ as the reduction of the set of all $(z,\zeta)\in \wh{E}\times \uu$ such that $f_{\wh{E}}(z)=f_\uu(\zeta)$, by the action of $K$. As a Lie algebroid, this is the action Lie algebroid $H\times \uu$ for the action  $\varrho\circ f_\uu$.  
\end{proof}
%
%Note that by construction, $L$ is embedded as a sub-Lie algebroid of $(\wh{E}\times \uu)\qu K$.
%
\begin{example}
The simplest examples of Harish-Chandra pairs with morphisms to $(\dd,H)$ are 
\[ (\dd,H),\ (\dd,\{e\}),\ (\g,\{e\}),\ (\h,H),\ (\h,\{e\}).\] The corresponding 
Lie algebroids with actions of $(H,E)$ are, respectively, 
\[ (\pt,\g),\ (H,\wh{E}),\ (H,E),\ (\pt,0),\ (H,TH).\] 
\end{example}

\begin{remark}
The constructions in this section can be generalized to the category $\ca{LG}^\vee$ of Lie groupoids, with morphisms the comorphisms of Lie groupoids. 
\end{remark}

Recall that a Lie algebroid structure on a vector bundle $E$ corresponds to a linear (equivalently, homogeneous of degree $-1$) Poisson structure on the dual bundle $E^*$. Under this correspondence, $\LA$-comorphisms $E_1\da E_2$ 
correspond to Poisson morphisms $E_1^*\to E_2^*$. Hence, all of the results above may be restated 
in Poisson-geometric terms. In particular, if $(H,E)$ is an $\LA^\vee$ Lie group then the total space of $E^*$ is a Poisson Lie group. (If $(H,E)$ is classified by the Lie algebra triple $(\dd,\g,\h)$, then the Manin triple of $E^*$ is 
$(\dd\ltimes\dd^*,\,\g\ltimes \h^*,\,\h\ltimes \g^*)$.) 

An $\LA^\vee$-action on $(M,L)$ becomes a Poisson action of $E^*$ on the Poisson manifold $L^*$, compatible with the fiberwise linear structure. Proposition \ref{prop:homla} classifies such actions when 
$H$ acts transitively on $M$; however, $E^*$ need not act transitively on $L^*$. 
\begin{proposition}\label{prop:addendum}
In the setting of Proposition \ref{prop:homla}, the action of $E^*$ on $L^*$ is transitive 
if and only if the map $f_\uu\colon \uu\to \dd$ is injective, with $f_\uu(\uu)\cap \h=\k$.  The classification of $\LA^\vee$-actions with this property is thus given by $K$-invariant Lie subalgebras $\cc\subseteq \dd$
with $\cc\cap\h=\k$. 
\end{proposition}
\begin{proof}
Since $H$ acts transitively on the base $M$, the action of $E^*$ on $L^*$ is transitive if and only if the action of $E^*|_e$ on $L^*|_{eK}$ is transitive. By dualizing the $\VB$-groupoid action of $E_e=\g\rra \g$ on $L_{eK}=\uz/\k$, we find that this is the action 
$\mu\circ \nu=\nu+f_\uu^*(\mu)$ for $\mu\in\g^*=\on{ann}_{\dd^*}(\h)$ and $\nu\in \on{ann}_{\uz^*}(\k)$. 
(This is well-defined, since $f_\uu(\k)\subseteq \h$ is equivalent to $f_\uu^*(\g^*)=f_\uu^*(\on{ann}_{\dd^*}(\h))\subseteq 
\on{ann}_{\uz^*}(\k)$.) This action is transitive if and only if $f_\uu^*\colon \dd^*\to \uu^*$ restricts to a surjective map $\on{ann}_{\dd^*}(\h)\to \on{ann}_{\uz^*}(\k)$, 
or equivalently if and only if $f_\uu$ induces an injective map $\uz/\k\to \dd/\h$. Since $f_\uu$ restricts to the identity map on $\k$, this is equivalent to $f_\uu$ itself being injective, with 
$f_\uu(\uu)\cap\h=\k$.  
\end{proof}

\section{General properties of Dirac actions}
In this Section, we will extend the results from the last section to Dirac actions. 
Throughout, we fix a Dirac-Lie group $(H,\,\AA,\,E)$.  In particular, $(H,E)$ is an $\LA^\vee$-Lie group, classified by an $H$-equivariant Lie algebra triple $(\dd,\g,\h)$. 

\subsection{Properties of the structure maps}
We will use the $\VB$-groupoid structure $\AA\rra \g$; recall that $E\rra \g$ is a vacant $\VB$-subgroupoid. 
Since the anchor map $\mathsf{a}_\AA\colon \AA\to TH$ is a $\VB$-groupoid morphism, the dual map $\mathsf{a}_\AA^*\colon T^*H\to \AA^*=\AA$ is a $\VB$-groupoid morphism, 
%
%\begin{equation}\label{eq:nicedia}
\[ 
 \xymatrix{ {T^*H} \ar[r]<2pt>\ar@<-2pt>[r] \ar[d]_{\mathsf{a}_\AA^*}&{\ \h^*}\ar[d]^{\varphi}\\
\AA \ar[r]<2pt>\ar@<-2pt>[r] & \g}
%\end{equation}  
\]
(See Appendix \ref{app:vb}.) The right vertical map in this diagram 
\begin{equation}\label{eq:varphi}
 \varphi\colon \h^*\to \g
 \end{equation}
is the map of units; it is characterized by the property
$\sz_\AA\circ \mathsf{a}_\AA^*=\varphi\circ \sz_{T^*H}$, with a similar 
property $\tz_\AA\circ \mathsf{a}_\AA^*=\varphi\circ \tz_{T^*H}$ for the target map. 
Let $\gamma_\AA\in \Gamma(S^2\AA)$ be the element defined by the metric. Then $\sz_\AA(\gamma_\AA)$ is a function on $H$ with values in $S^2\g$, and similarly for $\tz_\AA(\gamma_\AA)$. Let 
\[ \gamma_\g:=\tz_\AA(\gamma_\AA)|_e\in S^2\g.\]
\begin{lemma}
We have 
\[ \tz_\AA(\gamma_\AA)=-\sz_\AA(\gamma_\AA)=\gamma_\g.\]
In particular,  both sides 
are constant functions from $H$ to $S^2\g$. 
\end{lemma}
\begin{proof}
For all $x\in \AA$, with groupoid inverse $x^{-1}$, we have 
\[ \l x,x\r+\l x^{-1},x^{-1}\r=\l x\circ x^{-1},x\circ x^{-1}\r=
\l \tz_\AA(x),\,\tz_\AA(x)\r=0.\] 
 If  $x\in \ker(\sz_\AA|_e)=\on{ran}(\tz_\AA^*|_e)$ then $x^{-1}\in \ker(\tz_\AA|_e)=\on{ran}(\sz_\AA^*|_e)$, and 
 \[\l x,\xi\r+\l x^{-1},\xi\r=\l x\circ x^{-1},\ \xi\circ \xi\r
 =\l \tz_\AA(x),\xi\r=0
 \]
 for all $\xi\in\g$. 
Hence, if $x=\tz_\AA^*(\mu)|_e$ then 
$x^{-1}=-\sz_\AA^*(\mu)|_e$, and 
\[ \l \sz_\AA^*(\mu),\ \sz_\AA^*(\mu)\r|_e=\l x^{-1},x^{-1}\r=-\l x,x\r=-
\l \tz_\AA^*(\mu),\ \tz_\AA^*(\mu)\r|_e.\]
This proves $\sz_\AA(\gamma_\AA)|_e=-\tz_\AA(\gamma_\AA)|_e$. For all $h\in H$, the composition $\tz_\AA^*(\mu)|_e\circ 0_h$ is well-defined since $\on{ran}(\tz_\AA^*)_e=\ker(\sz_\AA)_e$. The composition  lies in  $\ker(\sz_\AA)_h=\on{ran}(\tz_\AA^*)_h$; indeed  $\tz_\AA^*(\mu)|_e\circ 0_h=\tz_\AA^*(\mu)|_h$, as we can see by taking the inner product with the identity $\tz_\AA(y)\circ y=y$, valid for all $y\in \AA_h$. Hence 
\[ \l \tz_\AA^*(\mu)|_h,\ \tz_\AA^*(\mu)|_h\r=
\l \tz_\AA^*(\mu)|_e\circ 0_h,\ \tz_\AA^*(\mu)|_e\circ 0_h\r=\l \tz_\AA^*(\mu)|_e,\ \tz_\AA^*(\mu)|_e\r,\] 
proving that $\tz_\AA(\gamma_\AA)$ is constant. Similarly, $\sz_\AA(\gamma_\AA)$ is constant.
\end{proof}

Consider a Dirac action of $(H,\,\AA,\,E)$ on a Dirac manifold $(M,\,\PP,\,L)$. In particular $H$ acts on $M$; let $\mathsf{a}_{M\times \h}\colon M\times \h\to TM$ be the infinitesimal action map. 
The dual map is the symplectic moment map $\uz_{T^*M}\colon T^*M\to \h^*$.  The anchor maps  intertwine the 
action of $\AA$ on $\PP$ with the action of $TH$ on $TM$. Hence, their duals intertwine the groupoid action of $T^*H\rra \h^*$ on $T^*M$ with the action of $\AA\rra \g$ on $\PP$. (See Appendix \ref{app:vb}.) In particular, the moment maps for the two groupoid actions are intertwined, giving a commutative diagram 
\[
\xymatrix{ {T^*M} \ar[r]^{\uz_{T^*M}}\ar[d]_{\mathsf{a}_\PP^*} &{\h^*}\ar[d]^{\varphi}\\
{\PP} \ar[r]_{\uz_\PP} & {\g}
}
\]
\begin{proposition}\label{prop:thishelps}
We have $\uz_\PP(\gamma_\PP)=\gamma_\g$, 
where $\gamma_\PP\in \Gamma(S^2\PP)$ is the element defined by the metric. 
\end{proposition}
\begin{proof}
Given $\mu \in \g^*$ and $m\in M$, we claim that
\begin{equation}\label{eq:aha}
\uz_\PP^*(\mu)_m=\tz_\AA^*|_e(\mu)\circ 0_m,\end{equation} 
with $0_m\in \PP_m$ the zero element. To see this, note first that the action is well-defined since $\on{ran}(\tz_\AA^*)=\ker(\sz_\AA)$. Furthermore, whenever
$x\in \AA_e,\ y\in \PP_m$ are composable 
elements, then
\[ \l x\circ y,\,\uz_\PP^*(\mu)\r=\l \uz_\PP(x\circ y),\,\mu\r=\l\tz_\AA(x),\,\mu\r=\l x,\,\tz_\AA^*(\mu)|_e\r.\]
This shows that the element $(\uz_\PP^*(\mu)_{m},\,\tz_\AA^*(\mu)|_e,0_m)$ is orthogonal to the graph of $\mathsf{a}_\PP$. Since $\on{Gr}(\mathsf{a}_\PP)^\perp=\on{Gr}(\mathsf{a}_\PP)$, this proves Equation \eqref{eq:aha}. Taking an inner product of this element with itself, we obtain
$\l \uz_\PP^*(\mu)_m,\ \uz_\PP^*(\mu)_m\r=\l \tz_\AA^*(\mu)|_e,\ \tz_\AA^*|(\mu)|_e\r$, proving $\uz_\PP(\gamma_\PP)=\gamma_\g$.  %Using \eqref{eq:aha}, we also find that 
%\[ \mathsf{a}_\PP(\uz_\PP^*(\mu)_m)=\mathsf{a}_\AA(\tz_\AA^*(\mu)|_e)\circ 0_m=\mathsf{a}_{M\times \h}|_m(\mathsf{a}_\AA(\tz_\AA^*(\mu)|_e))=\mathsf{a}_{M\times \h}|_m(\varphi^*(\mu))\]where now $0_m$ signifies the zero vector in $T_mM$, and $\circ$ is the action of $TH$ on $TM$. This shows $\mathsf{a}_{M\times \h}\circ \varphi^*=\mathsf{a}_\PP\circ \uz_\PP^*$, hence $\uz_\PP\circ \mathsf{a}_\PP^*=\varphi\circ \mathsf{a}_{M\times \h}^*$. 
\end{proof}

Using $\gamma_\g$ and $\varphi$, we define $\beta\in S^2\dd$ as the symmetric bilinear form on $\dd^*=\g^*\oplus \h^*$ given by 
\[ \beta(\mu_1,\mu_2)=\begin{cases}
\gamma_\g(\mu_1,\mu_2) & \mu_1,\mu_2\in \g^*,\\
\l\mu_1,\varphi(\mu_2)\r& \mu_1\in \g^*,\ \mu_2\in\h^*,\\
0 & \mu_1,\mu_2\in \h^*.
\end{cases}\]
By definition $\g\subseteq \dd$ is $\beta$-coisotropic, and $\pr_\g(\beta)=\gamma_\g$. In \cite{lib:dir} it was shown that $\beta$ is $\ad(\dd)$-invariant as 
well as $H$-invariant, and that the Dirac Lie group $(H,\AA,E)$ is classified by the $H$-equivariant Dirac-Manin triple $(\dd,\g,\h)_\beta$. Below we will give a conceptual explanation of these facts.

\subsection{The $H$-action on $\PP$}\label{subsec:bullet}
Consider a Dirac action of $(H,\,\AA,\,E)$ on a Dirac manifold $(M,\,\PP,\,L)$.  The action of $E\subseteq \AA$ on $\PP$ defines a  $\bullet$-action (cf.~ Section \ref{sec:general}) 
of $H$ on $\PP$, given as $h\bullet y=x\circ y$ 
where $x\in E_h$ is the unique element such that $\sz_\AA(x)=\uz_\PP(y)$. 
This action preserves the subbundle $L$, it satisfies 
$\uz_\PP(h\bullet y)=h\bullet \uz_\PP(y)$, 
and it preserves the metric on $\PP$, since 
\[\l h\bullet y,\ h\bullet y\r=
\l x\circ y,x\circ y\r=\l x,x\r+\l y,y\r=\l y,y\r.\] 
On the other hand, the $\bullet$-action preserves neither the Courant bracket on $\Gamma(\AA)$ nor the anchor, in general. 
\begin{examples}
\begin{enumerate}
\item 
For a Poisson action of a Poisson Lie group  $(H,\pi_H)$ on a Poisson manifold $(M,\pi_M)$, the corresponding $\bullet$-action on $\PP=\T M$ is the unique action that preserves the 
decomposition $\T M=TM\oplus \on{Gr}(\pi_M)$ as well as the metric on $\T M$, and that restricts to the tangent lift on $TM$. Note that this is different from the action $h\mapsto \T \mathsf{a}_M(h)$, which does not preserve $\on{Gr}(\pi_M)$ in general. 
\item Let $(\dd,\g,\h)_\beta=(\h\oplus \ol{\h},\,\h_\Delta,\h\oplus 0)_\beta$ defined by an $H$-invariant metric on $\h$ corresponding to the 
Cartan-Dirac structure $(\AA,E)=(H\times \dd,\,H\times \g)$, cf. Example \ref{ex:cartandirac}. The $\bullet$-action on $\g$ is trivial. Therefore, given a Dirac action on $(\PP,L)=(M\times \dd,\,M\times \mathfrak{l})$ the $\bullet$-action on $\PP$ is simply $h\bullet(m,\tau,\tau')=(h.m,\tau,\tau')$. 
\end{enumerate}
\end{examples}
\begin{theorem}\label{th:haction}
For any Dirac action of $(H,\AA,E)$ on $(M,\PP,L)$, the space $\Gamma(\PP)^H$  of $\bullet$-invariant sections is closed under the Courant bracket. For all $h\in H$ and $y\in \PP$, the difference $\mathsf{a}_\PP(h\bullet y)-h.\mathsf{a}_\PP(y)\in TM$ (where the dot indicates the tangent lift of the 
action) is tangent to $H$-orbits. 
Hence, if the action of $H$ on $M$ is a principal action, then the quotient space becomes a Dirac manifold 
$(M/H,\,\PP/H,\,L/H)$, with the quotient map defining a Dirac morphism. 
\end{theorem}
\begin{proof}
Define a Courant morphism $S\colon \AA\times \PP\da \PP$, 
with base map $H\times M\to M$ projection to the second factor, by declaring that $(x,y)\sim_S y'$ if and only if 
$y=y'$ and $x\in E$. A section $\sigma\in\Gamma(\PP)$ is $\bullet$-invariant if and only if there exists a section 
$\tilde{\sigma}\in \Gamma(\AA\times\PP)$ with $\tilde{\sigma}\sim_S \sigma$ and $\tilde{\sigma}\sim_{\mathsf{a}_\PP}\sigma$. (This section is necessarily unique.) 
Since these relations are preserved under Courant brackets (cf.~\eqref{eq:0}), it follows that the Courant bracket of two $\bullet$-invariant
sections is again invariant. 

For composable elements $x\in \AA$ and $y\in \PP$, we have 
$\mathsf{a}_\PP(x\circ y)=\mathsf{a}_\AA(x)\circ  \mathsf{a}_\PP(y)$, where the right hand side 
uses the action of $TH$ on $TP$. Writing $\mathsf{a}_\AA(x)=(h,\tau)$ in left trivialization, 
the difference $\mathsf{a}_\AA(x)\circ w-
h.w$ for any $w\in TM$ is tangent to $H$-orbits in $M$; in fact it is given by the infinitesimal action 
of $\Ad_h(\tau)$. 
%
% recall that the action of $TH$ on $TM$ is given in terms of left trivialization $TH=H\times \h$ by the formula 
%\[ (h,\tau).v=h.v-\mathsf{a}_{M\times \h}(\Ad_h \tau),\]
% for $h\in H,\ \tau\in\h$ and $v\in TM$. In our case, since $\sz(x)=\uz(y)$, we have that 
% $\mathsf{a}(x)=(h,\Ad_{h^{-1}}\pr_\h \Ad_h \uz(y))$. It follows that 
% \[ \mathsf{a}_\PP(h\bullet y)=\mathsf{a}_\PP(x\circ y)=\mathsf{a}_\AA(x).\mathsf{a}_\PP(y)=h.\mathsf{a}_\PP(y)-\mathsf{a}_{M\times \h}\big(\pr_\h(\Ad_h \uz(y))\big),\] 
%proving \eqref{it:d}. 
\end{proof}
The analogue to Proposition \ref{prop:uzl} holds true on invariant sections:
\begin{proposition}\label{prop:uzp}
For $\sigma_1,\sigma_2\in \Gamma(\PP)^H$, 
\begin{equation}\label{eq:uzformula}
 \uz_\PP(\Cour{\sigma_1,\sigma_2})=[\uz_\PP(\sigma_1),\uz_\PP(\sigma_2)]_\g
+\L_{\mathsf{a}_\PP(\sigma_1)}\uz_\PP(\sigma_2)-\L_{\mathsf{a}_\PP(\sigma_2)}\uz_\PP(\sigma_1);
\end{equation}
here $[\cdot,\cdot]_\g$ denotes the pointwise Lie bracket of $\g$-valued functions on $M$.  In particular, 
$\Gamma(\PP)^H\cap \ker(\uz_\PP)$ is closed under the Courant bracket. 
Furthermore, the space $\Gamma(\PP)^H\cap \ker(\uz_\PP)\cap \ker(\mathsf{a}_\PP)$ is an ideal in $\Gamma(\PP)^H$, relative to the Courant bracket. 
\end{proposition}
\begin{proof}
We use the notation from the proof of Theorem \ref{th:haction}.
The section $\wt{\sigma}$ associated to $\sigma\in\Gamma(\PP)^H$ has the form
$\wt{\sigma}=(\uz_\PP(\sigma),\sigma)$, where the first entry $\uz_\PP(\sigma)$ is regarded as a section of $\pr_H^*E=(H\times M)\times \g$ (depending trivially on the $H$-variable) and the 
second entry as a section of $\pr_M^*\PP$. The Courant bracket of two such sections 
$\wt{\sigma_i}=(\uz_\PP(\sigma_i),\sigma_i)$ is given by 
\[ \Cour{\wt{\sigma_1},\wt{\sigma_2}}=
\Big([\uz_\PP(\sigma_1),\uz_\PP(\sigma_2)]_\g+\L_{\mathsf{a}_\PP(\sigma_1)}\uz_\PP(\sigma_2)-\L_{\mathsf{a}_\PP(\sigma_2)}\uz_\PP(\sigma_1),\ \Cour{\sigma_1,\sigma_2}\Big).
\]
Since this section is related to $\sigma=\Cour{\sigma_1,\sigma_2}$ under both $\mathsf{a}_\PP$ and $S$
(cf. \eqref{eq:0})
it coincides with $(\uz_\PP(\Cour{\sigma_1,\sigma_2}),\Cour{\sigma_1,\sigma_2})$, proving \eqref{eq:uzformula}.

 Suppose $\sigma_1,\sigma_2\in \Gamma(\PP)^H$, with 
$\sig_1\in \ker(\uz_\PP)\cap \ker(\mathsf{a}_\PP)$. Then 
$\mathsf{a}_\PP(\Cour{\sigma_1,\sigma_2})=0$ since $\ker(\mathsf{a}_\PP)$ is an ideal for the Courant  bracket, while \eqref{eq:uzformula} shows that $\uz_\PP( \Cour{\sigma_1,\sigma_2})=0$. 
Similarly for $\Cour{\sigma_2,\sigma_1}$. 
\end{proof}
In particular, if the action of $H$ on $M$ is a principal action, and $\uz_\PP$ has constant rank, then 
$\ker(\uz_\PP)/H$ is an involutive subbundle of the Courant algebroid $\PP/H$. If 
$(\mathsf{a}_\PP,\uz_\PP)\colon \PP\to TM\times \g$ has constant rank, then the sections of $(\ker(\uz_\PP)\cap \ker(\mathsf{a}_\PP))/H$ form a Courant ideal. 
\begin{remark}
Note that the right hand side of formula \eqref{eq:uzformula} is skew-symmetric in $\sigma_1,\sigma_2$, hence 
so is the left hand side. This is possible since $\uz_\PP$ vanishes on $\Gamma(\on{ran}(\mathsf{a}_\PP^*))^H$, as 
one may also prove directly. 
%a consequence of Proposition \ref{prop:thishelps}.
\end{remark}

\subsection{The Lu-Dirac structure}
We next extend the construction of the Lu-Lie algebroid $\wh{L}$ to incorporate the Courant algebroid. 

\begin{theorem}\label{th:lu}
There is a canonically defined $H$-equivariant Dirac structure $(\wh{\PP},\wh{L})$ on $M$, such that 
\begin{equation}\label{eq:pdec}
\wh{\PP}=(M\times (\h\oplus\h^*))\oplus \PP,\ \ \ \wh{L}=(M\times \h)\oplus L\end{equation}
as (metrized) vector bundles, with anchor map $\mathsf{a}_{\wh{\PP}}(\tau+\nu,y)=\mathsf{a}_{M\times \h}(\tau)+\mathsf{a}_\PP(y)$. 
%
%\[\mathsf{a}_{\wh{\PP}}\colon \wh{\PP}\to TM,\ \ \ \mathsf{a}_{\wh{\PP}}(\tau+\nu,y)=\mathsf{a}_{M\times \h}(\tau)+\mathsf{a}_\PP(y).\]
%
The constant sections of $M\times\h$ are generators for the $H$-action on $\wh{\PP}$, and the induced $H$-action on $(M\times \h)^\perp/(M\times \h)\cong \PP$ is the $\bullet$-action on $\PP$. The map 
\begin{equation}\label{eq:fpmap}
 f_{\wh{\PP}}\colon \wh{\PP}\to \dd,\ \ \ f_{\wh{\PP}}(\tau+\nu,y)=\tau+\varphi(\nu)+\uz_\PP(y)\end{equation}
is an $H$-equivariant bundle morphism, with $ f_{\wh{\PP}}\circ  \mathsf{a}_{\wh{\PP}}^*=0$, 
compatible with brackets, and with 
\begin{equation}\label{eq:fpgamma}
 f_{\wh{\PP}}(\gamma_{\wh{\PP}})=\beta.\end{equation}
\end{theorem}
\begin{proof}
We will obtain $\wh{\PP}$ as a quotient $(\T H\times \PP)/H$. 
Consider the Dirac action of $(H,\AA,E)$ on $(H\times M, \T H\times \PP,\ TH\times L)$,
where $\mathsf{a}_{\T H\times \PP}$ is the composition of  
\[ \Mult_{\T H}\times \mathsf{a}_\PP\colon (\T H\times \AA)\times (\T H\times \PP)\da (\T H\times \PP)\]
with the Courant morphism $\AA\da \TH\times \AA,\ x\sim 
(\mathsf{a}_\AA(x)-\mu,\ x+\mathsf{a}_\AA^*(\mu))$ (with $\mu\in T^*H$) in the first factor. 
That is, for $x\in \AA_g$ and $w+\nu\in \T_hH,\ y\in \PP_h$, 
\begin{equation}\label{eq:bullaction1}x\circ (w+\nu,y)=\big( \mathsf{a}_\AA(x)\circ w+(-\mu)\circ \nu,\ 
(x+\mathsf{a}_\AA^*(\mu))\circ y
\big)
\end{equation}
whenever the composition on the right hand side is defined. In particular $\mu\in T^*_gH$ is the unique element such that $\nu':=(-\mu)\circ \nu$ (groupoid multiplication in $T^*H\rra \h^*$) is defined: 
thus $-\sz_{T^*H}(\mu)=\tz_{T^*H}(\nu)=\Ad_h \sz_{T^*H}(\nu)$, and in this case
$\nu'\in T^*_{gh}H$ is such that $\sz(\nu')=\sz(\nu)$. That is, $\nu'$ is the image of $\nu$ under the cotangent lift of $\mathsf{a}^L(g)$, while $-\mu$ is the image of $\nu$ under the cotangent lift of $\mathsf{a}^L(g)\circ \mathsf{a}^R(h)$. 
The condition that $(x+\mathsf{a}_\AA^*(\mu))\circ y$
be defined requires that 
\[ \sz_\AA(x)=\uz_\PP(y)-\sz_\AA(\mathsf{a}_\AA^*(\mu))
=\uz_\PP(y)-\varphi(\sz_{T^*H}(\mu))
=\uz_\PP(y)+\varphi(\tz_{T^*H}(\nu))
.\] 
This determines the moment map for the action: 
\begin{equation}\label{eq:momap1}
\uz_{\T H\times\PP}(w+\nu,y)=\uz_\PP(y)+\varphi(\tz_{T^*H}(\nu)).
\end{equation}
By construction, the $\VB$-groupoid action $\mathsf{a}_{\T H\times \PP}$ given by these formulas is a Courant morphism. 
To check that it is a Dirac action, let $(w',y')\in T_{gh}H\times L_{g.m}$ be given. Let $x\in E_g,\ y\in L_h$ be the unique elements such that 
$y'=x\circ y$, and put $w=\mathsf{a}_\AA(x)^{-1}\circ w'$. Then $(x,(w,y))\in E_g\times (TH\times L)_h$ is the unique element such that 
$x\circ (w,y)=(w',y')$.

The action of the subgroupoid $E\rra \g$ determines a $\bullet$-action 
of $H$ on $\T H\times \PP$, preserving $TH\times L$. 
By Theorem \ref{th:haction}, the quotient with respect to the 
$\bullet$-action is a Dirac structure $(\wh{\PP},\wh{L})$ over $(H\times M)/H\cong M$. As vector bundles, 
\[ \wh{\PP}=(\T H\times \PP)\big|_{\{e\}\times M}, \ \ \wh{L}=(TH\times L)\big|_{\{e\}
\times M}.\] 
Since $\T H=H\times (\h\oplus\h^*)$ by left-trivialization, this gives the decomposition \eqref{eq:pdec} as vector bundles. The description of the anchor is clear. 
The action $\T \mathsf{a}^R(h)\times \id$ on $\T H\times \PP$ is by Dirac automorphisms, and commutes with the $\bullet$-action, hence it descends to an action by Dirac automorphisms of $(\wh{\PP},\wh{L})$.

Note that the subbundle $TH\times \PP$ is invariant under the $\bullet$-action of $H$ as well as under the action $\T \mathsf{a}^R\times \id_\PP$. The diagonal $H$-action preserves 
the restriction to $\{e\}\times \PP$; it is given by 
$(w,y)\mapsto (\mathsf{a}(x)\circ w\circ h^{-1},\ h\bullet y)$. It follows that the 
induced action on $\PP$ is the $\bullet$-action. 
The map $\h\to \Gamma(\T H\times \PP),\ \tau\mapsto (\tau^L,0)$ defines generators for the $H$-action $h\mapsto \T \mathsf{a}^R(h)\times \id$ on $\T H\times \PP$. Its image under the quotient map are the constant sections of $M\times \h$, which hence are generators 
for the $H$-action on $\wh{\PP}$. 

To prove the properties of the map $f_{\wh{\PP}}$, consider the diagram
\[ \xymatrix@C=16ex{ {\T H\times \PP} \ar[r]^{(\mathsf{a}_{\T H\times \PP},\,\uz_{\T H\times \PP})}\ar[d] & {\ TH\times TM\times \g}\ar[d] \\
\wh{\PP} \ar[r]^{(\mathsf{a}_{\wh{\PP}},\ f_{\wh{\PP}})} & T M\times \dd
}\]
where the vertical maps are quotient maps under the $\bullet$-actions. This diagram commutes: 
The map $\mathsf{a}_{\T H\times \PP}\colon \T H\times \PP\to TH\times TM$, followed by 
the map $TH\times TM\to TM,\ (v,w)\mapsto v\circ w^{-1}$, descends to $\mathsf{a}_{\wh{\PP}}$, 
while the map $\uz_{\T H\times \PP}$, followed by the quotient map $TH\times \g\to \dd$, descends to $f_{\wh{\PP}}$. Proposition \ref{prop:uzp} shows that the upper horizontal map is bracket-preserving on 
$\bullet$-\emph{invariant} sections. Hence the lower horizontal map defines a bracket-preserving map on sections, that is, $f_{\wh{\PP}}$ is compatible with brackets: 
\begin{equation}\label{eq:fidentity}
f_{\wh{\PP}}(\Cour{\sigma_1,\sigma_2})=
[f_{\wh{\PP}}(\sigma_1),\ f_{\wh{\PP}}(\sigma_2)]+\L_{\mathsf{a}_{\wh{\PP}}(\sigma_1)}f_{\wh{\PP}}(\sigma_2)-\L_{\mathsf{a}_{\wh{\PP}}(\sigma_2)}f_{\wh{\PP}}(\sigma_1)
\end{equation} 
for all sections $\sigma_1,\sigma_2\in \Gamma(\wh{\PP})$. 
In particular, $f_{\wh{\PP}}\circ \mathsf{a}_{\wh{\PP}}^*=0$. 
Furthermore, since the upper horizontal map intertwines the action $\T \mathsf{a}^R\times \id_\PP$ 
with the action $T\mathsf{a}^R\times \id_{TM}\times \id_\g$, the lower horizontal map is $H$-equivariant, and in particular $f_{\wh{\PP}}$ is $H$-equivariant.

For $\mu\in \dd^*$, let $\mu=\mu'+\mu''$ be its decomposition into $\mu'\in \g^*$ and 
$\mu''\in\h^*$. The dual map to $f_{\wh{\PP}}$ at $m\in M$ satisfies 
$f_{\wh{\PP}}^*(\mu)=(\varphi^*(\mu'),\mu'',\uz_\PP^*(\mu')_m)$. 
It follows that
\[ \begin{split}
 \l f_{\wh{\PP}}^*(\mu),\ f_{\wh{\PP}}^*(\mu)\r
&=2 \l \phi^*(\mu'),\mu''\r+\l \uz_\PP^*(\mu')_m,\,\uz_\PP^*(\mu')_m\r\\
&=2 \l \phi(\mu''),\mu'\r+\gamma_\g(\mu',\mu')\\
&=\beta(\mu,\mu),
\end{split}
\]
proving $f_{\wh{\PP}}(\gamma_{\wh{\PP}})=\beta$.

\end{proof}

\begin{remark}
As remarked earlier, $\wh{L}=(M\times \h)\oplus L$ is a matched pair of Lie algebroids. 
Similarly, as we saw in the proof, $\wh{\PP}$ is the direct sum of two involutive subbundles, 
$M\times \h$ and $f_{\wh{\PP}}^{-1}(\g)=(M\times \h^*)\oplus \PP$. There is also a notion of \emph{matched pair of Courant algebroids}, due to 
Gr\"utzmann-Sti\'{e}non \cite{gru:mat}, but $\wh{\PP}$ does not fit into 
this framework, in general. 
\end{remark}

\begin{remark}
The stabilizer of the dressing action of $\dd$ at $h\in H$ is the $\beta$-coisotropic Lie subalgebra $\Ad_h(\g)$. Since $f_{\wh{\PP}}$ is compatible with brackets and satisfies 
\eqref{eq:fpgamma}, the discussion from Section \ref{subsec:pull1} shows that 
the vector bundle pullback of $\wh{\PP}$ under the projection $p\colon H\times M\to M$ is a Courant algebroid. One can show that $p^*\wh{\PP}\cong \T H\times \PP$. 
\end{remark}

The description of $\wh{\PP}$ simplifies in the special case that the Dirac-Manin triple 
$(\dd,\g,\h)_\beta$ is \emph{exact}: That is, $\beta$ is non-degenerate and $\g$ is Lagrangian. 
(As shown in \cite[Proposition 7.1]{lib:dir}, this condition is equivalent to exactness of the Courant algebroid $\AA$, i.e. 
$\ker(\mathsf{a}_\AA)=\on{ran}(\mathsf{a}_\AA^*)$ is a Lagrangian subbundle.)
\begin{proposition}
If the Dirac-Manin triple $(\dd,\g,\h)_\beta$ is exact, there is a canonical isomorphism 
\[ \wh{\PP}\cong \dd\times \PP\]
(product of Courant algebroids). Under this identification, $f_{\wh{\PP}}$ is projection to the first factor.  
\end{proposition}
\begin{proof}
Since $\beta$ is non-degenerate and $\g$ is Lagrangian,  the map $\varphi\colon \h^*\to \g$ is an isomorphism. Consequently, the map $f_{\wh{\PP}}$ is surjective, and 
by \eqref{eq:fidentity} $\ker(f_{\wh{\PP}})$ is an
$H$-invariant  involutive subbundle, which is a complement to $f_{\wh{\PP}}^{-1}(\g)^\perp=M\times \h^*$
inside $f_{\wh{\PP}}^{-1}(\g)$. The quotient map identifies this subbundle with 
$f_{\wh{\PP}}^{-1}(\g)/ f_{\wh{\PP}}^{-1}(\g)^\perp=\PP$ as a Courant algebroid. 

The map $f_{\wh{\PP}}$ defines a trivialization $\ker(f_{\wh{\PP}})^\perp=M\times \dd$. 
For a \emph{constant} section $\sigma_1$ of $M\times \dd$, and any section $\sigma$ of $\ker(f_{\wh{\PP}})$, we obtain, using \eqref{eq:fidentity},  
\[ f_{\wh{\PP}}(\Cour{\sigma_1,\sigma})=-\L_{\mathsf{a}_\PP(\sigma)}f_{\wh{\PP}}(\sigma_1)=0\] 
hence 
$\Cour{\sigma_1,\sigma}\in \Gamma(\ker(f_{\wh{\PP}}))$. If $\sigma_1,\sigma_2$ are two \emph{constant} sections of $M\times \dd$, and $\sigma\in \Gamma(\ker(f_{\wh{\PP}}))$, 
we see that 
$\l \Cour{\sigma_1,\sigma_2},\sigma\r=-\l \sigma_2,\Cour{\sigma_1,\sigma}\r=0$. 
Hence $\Cour{\sigma_1,\sigma_2}$ is again a section of $M\times \dd$, 
and since $f_{\wh{\PP}}(\Cour{\sigma_1,\sigma_2})=[f_{\wh{\PP}}(\sigma_1),f_{\wh{\PP}}(\sigma_2)]$ this sections is again constant. 
This shows that the Courant algebroid $\wh{\PP}$ is the product of the metrized Lie algebra $\dd$ (embedded as constant sections) with $\PP$ (embedded as $\ker(f_{\wh{\PP}})$). 
\end{proof}

\begin{remark} After a calculation, one finds that 
the isomorphism from $\dd\times \PP_m$ to  $\wh{\PP}_m=(\h\oplus \h^*)\oplus \PP_m$ 
is given by 
\[  (\lambda,y)\mapsto \big(\tau+\nu-\varphi^{-1}(\uz_\PP(y)),\ y+\uz_\PP^*(\phi^{-1})^*\tau \big), 
\]
where $\tau\in\h$ and $\nu\in\h^*$ are determined by 
$\tau=\pr_\h(\lambda),\ \phi(\nu)=(1-\pr_{\h^\perp})(\lambda)$. The inclusion 
of $\wh{L}$ reads as, 
\[ \wh{L}_m\to \dd\times \PP_m,\ (\tau,z)\mapsto \big(\tau+\uz_\PP(z),\,z-\uz_\PP^*(\phi^{-1})^*\tau\big).\] 
As a special case, we recover a result of Bursztyn-Crainic-\v{S}evera  \cite{bur:qua}, 
according to which for any Poisson action $(H,\pi_H)\times (M,\pi_M)\to (M,\pi_M)$,   the Lu-Lie algebroid $\wh{L}=(M\times \h) \oplus T_\pi^* M$ may be realized as a Dirac structure in the direct product $\dd \times \T M$. 
%Note however that the theory of Bursztyn-Crainic-\v{S}evera  applies more generally to quasi-Poisson actions. 
\end{remark}

%The same argument as for $\wh{L}$, replacing Lie algebroid brackets with Courant algebroid brackets, shows: 
%\begin{lemma}\label{lem:invt}The space $\Gamma((M\times \h)\oplus \PP)^H$ is closed under the Courant algebroid bracket, and \begin{equation}\label{eq:proj2}
 %\Gamma((M\times \h)\oplus \PP)^H/\Gamma(M\times \h)^H\to \Gamma(\PP)^H \pr_\PP\colon \Gamma((M\times \h)\oplus \PP)^H \to \Gamma(\PP)^H \end{equation} (where the invariants in $\Gamma(\PP)$ are relative to the $\bullet$-action) is bracket preserving. \end{lemma}

 \subsection{The action $\mathsf{a}_\q$}\label{subsec:aq}
Put $\q:=\AA_e,\ \g:=E_e$, and let $\gamma_\q\in S^2\q$ be the element defined 
by the metric. The $\VB$-groupoid structure of $\AA$ restricts to a $\VB$-groupoid structure of its unit fiber 
\begin{equation}\label{eq:qgroupoid}
 \q\rra \g.\end{equation}
The groupoid structure is compatible with the metric, in the sense that 
$\l \lambda_1\circ \lambda_2,\lambda_1\circ \lambda_2\r=\l\lambda_1,\lambda_1\r+\l\lambda_2,\lambda_2\r$ for composable elements $\lambda_1,\lambda_2$. 
Let $\rr=\ker(\tz_\q)\subseteq \q$, so that $\rr^\perp=\ker(\sz_\q)$. 
Denote by $\pr_{\rr},\ \pr_{\rr^\perp}=1-\pr_\rr^*\in\End(\q)$ the projections to $\rr,\rr^\perp$ with kernel $\g$. Then 
\[ \sz_\q(\lambda)=(1-\pr_{\rr^\perp})(\lambda),\ \ \tz_\q(\lambda)=(1-\pr_{\rr})(\lambda),\]
and  the groupoid multiplication is uniquely determined by its compatibility with the metric: 
\[ \lambda_1\circ \lambda_2=\lambda_2+\pr_{\rr^\perp}(\lambda_1),\]
see the proof of Lemma \ref{lem:why} below. By applying Theorem \ref{th:haction} to the action $\mathsf{a}^L_\AA$ of $(H,\,\AA,\,E)$ on itself, where
\[(x,y)\sim_{\mathsf{a}^L_\AA} x\circ y,\] 
 it follows that $(\AA/H,\,E/H)=(\q,\,\g)$  is a Dirac structure over 
$H/H=\pt$. Thus, $(\q,\g)_{\gamma_\q}$ is a Manin pair.  By Proposition \ref{prop:uzp},  $\ker(\tz_\AA)/H=\ker(\tz_\AA|_e)=\rr$ 
is a Lie subalgebra of $\q$ complementary to $\g$. 
(The space $\ker(\sz_\AA|_e)=\rr^\perp$ is not a Lie subalgebra, in general.) 
The Dirac action $\mathsf{a}^R_\AA$ of $(H,\,\ol{\AA},\,E)$
on $(H,\AA,E)$,  given as 
\[ (x,y)\sim_{\mathsf{a}^R_\AA} y\circ x^{-1}\]  
commutes with $\mathsf{a}^L_\AA$, hence it descends to a Dirac action $\mathsf{a}_\q$ of $(H,\,\ol{\AA},\,E)$
on $(\pt,\q,\g)$, with moment map the source map $\sz_\q$.
The resulting $\bullet$-action of $H$ on $\q$ is the unique extension of the $\bullet$-action on $\g$
that preserves the metric and for which $\sz_\q$ (hence also $\tz_\q$) is equivariant. 
It is hence an action by automorphisms of the metrized groupoid $\q\rra \g$. 

By applying Theorem \ref{th:lu} to the Dirac action $\mathsf{a}_\q$, we obtain a Manin pair $(\wh{\q},\wh{\g})_{\gamma_{\wh{\q}}}$, with an action of $H$ by automorphisms and 
with an $H$-equivariant Lie algebra morphism $\bar{f}_{\wh{\q}}\colon \wh{\q}\to \dd$. Here the bar serves as a reminder that we are using a Dirac action of 
$(H,\ol{\AA},E)$ (with the opposite metric); by Theorem \ref{th:lu} this implies in particular that
\begin{equation}\label{eq:09}
\bar{f}_{\wh{\q}}(\gamma_{\wh{\q}})=-\beta,\end{equation}
with a minus sign. 
\begin{proposition}\label{prop:atlast}
$(\dd,\g,\h)_\beta$ is an $H$-equivariant Dirac-Manin triple, and 
\[ (\wh{\q},\,\wh{\g})_{\gamma_{\wh\q}}=(\dd\ltimes\dd^*_\beta,\,\dd)_{\wt\beta}.\]
Here $\dd^*_\beta$ is embedded as 
$\ker(\bar{f}_{\wh{\q}})=\on{ran}(\bar{f}_{\wh{\q}}^*)^\perp$. 
\end{proposition}
\begin{proof}
Since the Lie algebra morphism $\bar{f}_{\wh{\q}}\colon \wh{\q}\to \dd$ restricts to the identity map on $\wh{\g}=\dd$, it is $\ad(\dd)$-equivariant as well as $H$-equivariant. 
 From this, it is immediate that $\beta$ is  $H$-invariant as well as $\ad$-invariant. (In \cite{lib:dir}, this was proved by direct calculations.)
The kernel $\ker(\bar{f}_{\wh{\q}})$ is an $H$-invariant ideal complementary to $\dd$, hence it may be identified with the dual space to $\dd$. 
On the orthogonal space $\ker(\bar{f}_{\wh{\q}})^\perp=\on{ran}(\bar{f}_{\wh{\q}}^*)$, the metric 
restricts to $-\beta$, by definition. Hence the restriction to $\ker(\bar{f}_{\wh{\q}})$
is $+\beta$. It follows that $\ker(\bar{f}_{\wh{\q}})=\dd^*_\beta$ as a 
Lie algebra as well as a $\dd$-module, and hence 
$\wh{\q}=\dd\ltimes\dd^*_\beta$ as a metrized Lie algebra.
%The subspace $\on{ann}(\g)\subseteq \dd^*_\beta\subseteq \wh{\q}$ is 
%ust $\h^*\subseteq \wh{\q}$, and in particular is isotropic. This means that $\beta$ restricts to $0$ on $\on{ann}(\g)$. That is, $\g$ is $\beta$-coisotropic. 
\end{proof}
Let $f_{\wh{\q}}\colon \wh{\q}\to \dd$ be the projection along $\ker(\bar{f}_{\wh{\q}})^\perp$. Since the latter is an ideal, the map $f_{\wh{\q}}$ is a Lie algebra morphism, 
and by the proof of Proposition \ref{prop:atlast} we have that
\[ f_{\wh{\q}}(\gamma_{\wh{\q}})=\beta.\]
The Manin pair $(\q,\g)_{\gamma_\q}$ is recovered from 
$(\dd\ltimes\dd^*_\beta,\,\dd)_{\ti{\beta}}$ 
as the reduction 
\[ \q=\bar{f}_{\wh{\q}}^{-1}(\g)/\bar{f}_{\wh{\q}}^{-1}(\g)^\perp,\ \ 
\g=(\bar{f}_{\wh{\q}}^{-1}(\g)\cap\dd)/(\bar{f}_{\wh{\q}}^{-1}(\g)^\perp\cap\dd)
.\]
(See also Section \ref{subsec:dmp}.) The map $f_{\wh{\q}}$ vanishes on 
$\bar{f}_{\wh{\q}}^{-1}(\g)^\perp\subseteq \bar{f}_{\wh{\q}}^{-1}(0)^\perp
=\ker(f_{\wh{\q}})$, hence it descends to a Lie algebra morphism $f_\q\colon \q\to \dd$,

\begin{proposition}
Given a Dirac action of $(H,\AA,E)$ on $(M,\PP,L)$, there is a canonical Dirac comorphism 
\[ (\PP,L)\da (\q,\g),\]
defined by the property that $y\sim \zeta$ if and only if there exists $z\in L$ with 
$y=\zeta\circ z$. It is equivariant for the action of $(H,\AA,E)$. 
\end{proposition}
\begin{proof}
The relation given above is obtained from the action $\mathsf{a}_\PP\colon \AA\times \PP\da \PP$
together with the morphism $\AA\to \q$ (given by trivialization) and $\PP\to 0$ (given by $L$). 
As a composition of Courant relations, it is again a Courant relation. It is a Dirac comorphism, since for 
$y\in L$, the unique element $\xi\in \g$ with $y\sim \xi$ is given by $\xi=\uz_L(y)$. 
\end{proof}

\section{Classification results for Dirac actions}
We will now use the results of the previous section to classify the Dirac actions on 
Dirac manifolds of the form $(H/K,\PP,L)$.

\subsection{The Dirac structure $(H\times (\dd\ltimes \dd^*_\beta),\,H\times \dd)$}
By definition of $(\q,\g)$ as the quotient of $(\AA,E)$ under the $\bullet$-action for 
$\mathsf{a}^L_\AA$, one obtains a  trivialization 
\begin{equation}\label{eq:trivA} (\AA,E)=(H\times \q,\,H\times\g),\end{equation}
in such a way that the constant sections of $H\times \q$ are the $\bullet$-invariant sections of $\AA$. It identifies the Dirac Lie group structure as an \emph{action  Dirac structure}, for some action of $\q$ on $H$ that will be determined below. 

The Dirac morphism $(\AA,E)\to (\q,\g)$ defined by the trivialization 
is equivariant for the Dirac action $\mathsf{a}^R_\AA$ of $(\ol{\AA},E)$ (with base action $\mathsf{a}^R_H$). Hence it extends to 
an $H$-equivariant Dirac morphism 
\[ (\wh{\AA},\,\wh{E})\to (\wh{\q},\wh{\g})
=(\dd\ltimes \dd^*_\beta,\,\dd)
.\]
Since both sides have the same rank, this morphism is given by an actual vector bundle map, 
$\wh{\AA}\to \dd\ltimes \dd^*_\beta$, which is a fiberwise isomorphism. 
This defines a trivialization 
\begin{equation}\label{eq:trivialization}
 (\wh{\AA},\,\wh{E})=(H\times (\dd\ltimes \dd^*_\beta),\ H\times \dd).
 \end{equation} 
\begin{proposition}\label{prop:ahataction}
The trivialization \eqref{eq:trivialization} identifies $(\wh{\AA},\wh{E})$ with the action Dirac structure 
for the Manin pair $(\wh{\q},\wh{\g})_{\gamma_{\wh{\q}}}$ and the action 
$\varrho\circ f_{\wh{\q}}\colon \wh{\q}\to \Gamma(TH)$. 
\end{proposition}
\begin{proof}
The general construction gives a bracket preserving map $\bar{f}_{\wh{\AA}}\colon \wh{\AA}\to \dd$ with $\bar{f}_{\wh{\AA}}(\gamma_{\wh{\AA}})=-\beta$. 
(In terms of the trivialization, this is the projection to $\dd$.)
Both $\ker(\bar{f}_{\wh{\AA}})$ and $\ker(\bar{f}_{\wh{\AA}})^\perp=\on{ran}(\bar{f}_{\wh{\AA}}^*)$ are complements to the Lagrangian subbundle $\wh{E}\subseteq \wh{\AA}$. Since $\bar{f}_{\wh{\AA}}\circ 
\mathsf{a}_{\wh{\AA}}^*=0$, we have $\mathsf{a}_{\wh{\AA}}\circ \bar{f}_{\wh{\AA}}^*=0$. 
That is, the anchor map vanishes on the second summand of the decomposition
\[ \wh{\AA}=\wh{E}\oplus \ker(\bar{f}_{\wh{\AA}})^\perp.\]
In terms of the trivialization, $\ker(\bar{f}_{\wh{\AA}})^\perp=H\times \ker(f_{\wh{\q}})$. But we had already shown that $\wh{E}$ is an action Lie algebroid, i.e. its anchor is given by the restriction of $\varrho\circ f_{\wh{\q}}$.  
\end{proof}

Since the anchor of 
$\AA\subseteq \wh{\AA}$ is obtained by restriction, we conclude that $(\AA,E)$ 
is an action Dirac structure for the action $\varrho\circ f_\q\colon \q\to \Gamma(TH)$. 
As explained in \cite{lib:dir} the groupoid structure of $\AA=H\times \q$ 
is that of a semidirect product (see Appendix \ref{app:semidirect}) of $H$ with the groupoid $\q\rra \g$, using the 
$\bullet$-action of $H$ on $\q$. \vskip.1in 

To summarize: Given a Dirac Lie group $(H,\AA,E)$, one obtains an $H$-equivariant Dirac-Manin triple $(\dd,\g,\h)_\beta$. The Dirac-Manin pair $(\dd,\g)_\beta$ determines a Main pair 
$(\q,\g)_\gamma$ as in Section \ref{subsec:dmp}. The Dirac structure on $(\AA,E)$ on $H$ is that 
of an action Dirac structure $(H\times \q,\ H\times \g)$ for 
$\varrho\circ f_\q$, while the $\VB$-groupoid structure $\AA\rra \g$ is that of a semi-direct product 
of $H$ with \eqref{eq:qgroupoid}. This proves that up to isomorphism, the Dirac Lie group 
$(H,\AA,E)$ is uniquely determined by its Dirac-Manin triple. As shown in \cite{lib:dir}, any Dirac-Manin triple arises in this way.

\subsection{Homogeneous spaces}
We will now consider Dirac actions of $(H,\AA,E)$ on $(M,\PP,L)$, with $M=H/K$ a homogeneous space. 
Restricting to the group unit $e\in H$, the groupoid action 
of $\AA\rra \g$ on $\PP$ gives an action of $\q\rra \g$ on 
$\pp=\PP_{eK}$, with moment map $\uz_\pp\colon \pp\to \g$. This groupoid action 
is compatible with the metrics, and is uniquely determined by this property and 
the moment map  $\uz_\pp$. Indeed, letting $F_\pp\colon \q\to \pp$ be the map 
defined by $\l F_\pp(\lambda),z\r=\l \lambda,\,\uz_\pp(z)\r$, we have
\begin{lemma}\label{lem:why}
The groupoid action of $\q\rra \g$ on $\pp$ is given by $\lambda\circ z=z+F_\pp(\lambda)$. 
\end{lemma}
\begin{proof}
For all $y\in \pp$, we have $\uz_\pp(y)\circ y=y$. Taking the inner product with 
$\lambda\circ z=z'$, we see that $\l z',y\r=\l z,y\r+\l \lambda,\,\uz_\pp(y)\r
=\l z+F_\pp(\lambda),\,y\r$, hence $z'=z+F_\pp(\lambda)$. 
\end{proof}
Use  the $\bullet$-action of $H$ on $\PP$ to write $\PP=H\times_K\pp$, and 
write $\AA=H\times \q$. 
\begin{lemma}\label{lem:act}
The groupoid action of $\AA\rra \g$ on $\PP$ is given by 
$\uz_\PP([(g,z)])=g\bullet \uz_\pp(z)$ and
\[ (h,\lambda)\circ [(g,z)]=[(hg,z+F_\pp(g^{-1}\bullet \lambda))].\]
\end{lemma}
\begin{proof}
Note that $F_\pp$ is $K$-equivariant relative to the $\bullet$-action of $K\subseteq H$.
Hence, the action of $K$ on $\pp$ is compatible with the structure as a 
$\VB$-groupoid module over $\q\rra \g$: 
\[ k.(\lambda\circ z)=
(k\bullet \lambda)\circ k.z.\]
We conclude that the $\VB$-groupoid action of $\AA\rra \g$ on $\PP$ 
is that of the semidirect product $H\ltimes \q\rra \g$ on $H\times_K \pp$, as in Appendix \ref{app:semidirect}. 
\end{proof}
The following example will be important in what follows. 

\begin{example}
Similar to Proposition \ref{prop:ehate}, the multiplication morphism $\Mult_\AA$ lifts to a Dirac action of $(H,\AA,E)$ on $(H,\wh{\AA},\wh{E})$, defining in particular a $\bullet$-action of $H$ on $\wh{\AA}$ which commutes with the action of $H$ by Dirac automorphisms. In terms of 
the trivialization $\wh{\AA}=H\times \wh{\q}$, the $\bullet$-action reads as 
$h\bullet(g,z)=(hg,z)$, while the $H$ action by automorphisms is $h.(g,z)=(gh^{-1},h.z)$, with the given $H$-action on $\wh{\q}=\dd\ltimes \dd^*_\beta$. 
Proposition \ref{prop:ahataction} shows that $(\wh{\AA},\wh{E})$ is an action Dirac structure for $\varrho\circ f_{\wh{\q}}$. On the other hand, as a $\VB$-groupoid module, $\wh{\AA}$  is fully determined by the moment map 
$\uz_{\wh{\AA}}|_e=\uz_{\wh{\q}}$ at the identity fiber. The latter  is induced 
by the target map $\tz_\q=(1-\pr_\rr)=\pr_\g \circ f_\q$, hence 
\[ \uz_{\wh{\q}}=\pr_\g \circ f_{\wh{\q}}\colon \wh{\q}\to \g.\]
With $F_{\wh{\q}}\colon \q\to \wh{\q}$ defined as above (for $\pp=\wh{\q}$), Lemma \ref{lem:act} shows  
\[ (h,\lambda)\circ (g,z)=(hg,\,z+F_{\wh{\q}}(g^{-1}\bullet\lambda))\]
for composable elements $(h,\lambda)\in \AA$ and $(g,z)\in \wh{\AA}$.  
\end{example}

Returning to Dirac structures over homogeneous spaces, we have the following classification result. 
Let $\phi\colon K\to H$ denote the inclusion. 
\begin{theorem}[Classification of Dirac Lie group actions on homogeneous spaces] \label{th:class}The Dirac structures $(\PP,L)$ 
on $H/K$, together with Dirac actions of $(H,\AA,E)$, are classified by 
the following pieces of data: 

\begin{enumerate}
\item[(i)] A Harish-Chandra pair $(\nn,K)$ together with a non-degenerate $K$-invariant element $\gamma_\nn\in S^2\nn$, 
\item[(ii)] a morphism of Harish-Chandra pairs $(f_\nn,\phi)\colon (\nn,K)\to (\dd,H)$ 
such that $f_\nn(\gamma_\nn)=\beta$,
\item[(iii)] a $K$-invariant Lagrangian Lie subalgebra $\uu\subseteq \nn$. 
\end{enumerate}
%
%\begin{enumerate}\item[(i)] A Manin pair $(\nn,\uu)_{\gamma_\nn}$ with an action of $K$ by automorphisms,\item[(ii)] an embedding $\k\hra \uu$ as a $K$-invariant Lie subalgebra, making $(\nn,K)$ into a Harish-Chandra pair,   \item[(iii)] a morphism of Harish-Chandra pairs $(f_\nn,\phi)\colon (\nn,K)\to (\dd,H)$ such that $f_\nn(\gamma_\nn)=\beta$. \end{enumerate}
\end{theorem}
\begin{remark}
We stress again (cf.~ \ref{subsec:jotz}) that Theorem \ref{th:class} is different from the classification of `Dirac homogeneous spaces' of Jotz \cite{jot:dir}. Theorem \ref{th:class} does not imply those results, or vice versa. 
\end{remark}

The proof of Theorem \ref{th:class}
is divided in several stages. First, we show how to associate to a given 
$(\PP,L)$ the data (i),(ii),(iii). Second, we give a normal form for $(\PP,L)$ and the action $\mathsf{a}_\PP$ in terms of these data, thus proving that $(\PP,L)$ is uniquely determined by the data. 
Third, we show that any given set of data (i),(ii),(iii) arises in this way. 

\subsubsection{Construction of the data (i),(ii),(iii)}
We assume that the Dirac structure 
$(\PP,L)$ on $H/K$ and the Dirac action $\mathsf{a}_\PP$ of $(H,\AA,E)$ are given. Let $(\wh{\PP},\wh{L})$ and $f_{\wh{\PP}}\colon \wh{\PP}\to \dd$ be as in  Theorem \ref{th:lu}. Since the anchor map for $\wh{\PP}$ is surjective, its kernel   $\ker(\mathsf{a}_{\wh{\PP}})$ is an involutive coisotropic subbundle, with 
$\ker(\mathsf{a}_{\wh{\PP}})^\perp=\on{ran}(\mathsf{a}_{\wh{\PP}}^*)$. The quotient is an $H$-equivariant Courant algebroid over $H/K$ containing $\ker(\mathsf{a}_{\wh{L}})$ as a Lagrangian subbundle. Both have the zero 
anchor map, and are hence bundles of Lie algebras. Their fibers at $eK$ are $K$-equivariant Lie algebras. 
We denote these by  $\nn,\ \uu$, so that
\[  \ker(\mathsf{a}_{\wh{\PP}})/\on{ran}(\mathsf{a}_{\wh{\PP}}^*)=H\times_K \nn,\ \ \ 
\ker(\mathsf{a}_{\wh{L}})=H\times_K \uu.
\]
Let $\gamma_\nn\in S^2\nn$ be the element defined by the metric. 
Then $(\nn,\uu)_{\gamma_\nn}$ is a Manin pair, with an action of $K$ by Manin pair automorphisms.
The inclusion of $\k\subseteq \uu$ into $\nn$ defines generators for the $K$-action; hence 
$(\nn,K)$ is a Harish-Chandra pair. 
Since $f_{\wh{\PP}}$ vanishes on $\on{ran}(\mathsf{a}_{\wh{\PP}}^*)$, it defines an $H$-equivariant map $H\times_K \nn \to \dd$. The restriction to $eK$ is a $K$-equivariant map, 
\[ f_\nn\colon \nn \to \dd.\]
Since $f_{\wh{\PP}}$ is compatible with brackets and satisfies $f_{\wh{\PP}}(\gamma_\PP)=\beta$,  the map 
$f_\nn$ is a Lie algebra morphism with $f_\nn(\gamma_\nn)=\beta$. Note that 
$f_\nn$ intertwines the inclusions of $\k$ into $\nn$ and into $\h\subseteq \dd$, hence 
it defines a morphism of Harish-Chandra pairs $(f_\nn,\phi)\colon (\nn,K)\to (\dd,H)$, 
where $\phi\colon K\to H$ is the inclusion.

\subsubsection{Construction of a normal form} \label{subsec:normal}
Consider first the case $K=\{e\}$, thus $M=H$. 
%We will use the following notation, for any metrized Lie algebra $\nn$ with a morphism $f_\nn\colon \nn\to \dd$ satisfying $f_\nn(\gamma)=\beta$. 
Let 
\[ F_\nn\colon \q\to \nn\]
be the map defined by 
\[ \l F_\nn(\lambda),\zeta\r =\l \lambda,\ \pr_\g(f_\nn(\zeta))\r,\]
for all $\lambda\in\q,\ \zeta\in\nn$. Note that $F_\nn$ vanishes on $\g\subseteq\q$, and the induced map $\g^*=\q/\g\to \n$ is the restriction of $f_\nn^*\colon \dd^*\to \nn$ to the subspace $\g^*\subseteq \dd^*$.  As before, we denote by $\varrho\colon \dd\to \Gamma(TH)$ the dressing action \eqref{eq:action}.

\begin{proposition}\label{prop:A}
Suppose $K=\{e\}$, thus $M=H$. There is a canonical identification of  
$(\PP,L)$ with the action Dirac structure $(H\times \nn,H\times \uu)$, 
where $\nn$ acts by the composition $\varrho\circ f_\nn$. The structure as a module over the $\VB$-groupoid 
$\AA=H\times \g\rra \g$ is given by the moment map 
\[ \uz_\PP(h,\zeta)=h\bullet \pr_\g (f_{\nn}(\zeta))\]
and the formula, for  composable elements $(g,\lambda)\in \AA$ and $(h,\zeta)\in \PP$, 
\[ (g,\lambda)\circ (h,\zeta)=(gh,\ \zeta+ F_{\nn}(h^{-1}\bullet \lambda)).\]
\end{proposition}
\begin{proof}
We will regard $\PP$ as the second summand of $\wh{\PP}=H\times (\h\oplus \h^*)\oplus \PP$. Given
$y\in \PP_h$, let $\tau\in\h$ be the unique element 
such that  $\mathsf{a}_{H\times \h}(h,\tau)=\mathsf{a}_\PP(y)$. Write $y''=(h,\tau)$ and put
$y'=y-y''\in \wh{\PP}$. Then $\mathsf{a}_{\wh{\PP}}(y')=0$, and the  map taking $y$ to $y'$ defines an isomorphism  of metrized vector bundles
\[ \PP\cong \ker(\mathsf{a}_{\wh{\PP}})/\on{ran}(\mathsf{a}_{\wh{\PP}}^*)=H\times \nn.\]
It restricts to the isomorphism $L\cong H\times \uu$ from Section \ref{subsec:homla}. 
Let $(h,\zeta)\in H\times \nn$ be the element corresponding to $y\in\PP$ under this isomorphism. Repeating the argument for $\wh{L}$ in the proof of Proposition \ref{prop:homla}, 
we obtain
\[ \uz_\PP(y)=f_{\wh{\PP}}(y)=f_{\wh{\PP}}(y')+f_{\wh{\PP}}(y'')=\Ad_h f_{\nn}(\zeta)+\tau,\]
Projecting to the $\g$-component, it follows that 
$\uz_\PP(h,\zeta)=h\bullet \pr_\g (f_{\nn}(\zeta))$. Projecting to the $\h$-component, 
we find that 
\[ \mathsf{a}_\PP(h,\zeta)=\big(h,\Ad_{h^{-1}}\pr_\h \Ad_h f_\nn(\zeta)\big),\]
which identifies $(\PP,L)$ as an action Dirac structure for $\varrho\circ f_\nn$. 
The formula for the-$\VB$-groupoid action is determined by $\uz_\PP$, by the formulas 
from  Lemmas \ref{lem:why} and \ref{lem:act}. 
\end{proof}
For general $M=H/K$, let $p\colon H\to H/K$ be the projection. The proposition above applies to the action of $(H,\AA,E)$ on $(H,\,p^!\PP,\,p^!L)$. This action commutes with the action of $K$ (with base action $k\mapsto \mathsf{a}^R(k)$) by Dirac automorphisms, 
and $(\PP,L)=(p^!\PP\qu K,\,p^!L\qu K)$. In terms of the identification $\PP\cong H\times \nn$, the  $K$-action reads as 
\[ k.(h,\,\zeta)=(hk^{-1},\,k.\zeta).\]
We thus have 
\[ (\PP,L)=(H\times \nn\qu K,\ H\times \uu\qu K)
=(H\times_K \pp,\ H\times_K \mathfrak{l})
\]
with $\pp=\k^\perp/\k,\ \ \mathfrak{l}=\uu/\k$. The map $\uz_{p^!\PP}$ descends to the map 
\[ \uz_\PP\colon  H\times \nn\qu K\to \g,\ \ \uz_\PP([(h,\zeta)])=h\bullet \pr_\g (f_{\nn}(\zeta)),\]
and the groupoid action of $\AA=H\times \q$ descends to a well-defined groupoid action 
\begin{equation}\label{eq:graction}
 (g,\lambda)\circ [(h,\zeta)]=[(gh,\ \zeta+ F_{\nn}(h^{-1}\bullet \lambda))].\end{equation}
In summary, we see that the Dirac structure $(\PP,L)$, and the Dirac action of 
$(H,\AA,E)$ on $(H/K,\PP,L)$, are fully determined by the data (i),(ii),(iii).

\subsubsection{Construction of $(M,\PP,L)$ }
Given the data (i),(ii),(iii), we may use the formulas above to define $(M,\PP,L)$. 
That is, as a Dirac structure $(\PP,L)$ is the
reduction $(H\times \nn\qu K,\ H\times \uu\qu K)$  of the action Dirac structure for the action $\varrho\circ f_\nn$. 
As a $\VB$-groupoid module over $\AA$, it is the module $H\times_K \pp$ 
over the semi-direct product $H\ltimes \q\rra \g$, where the groupoid action of 
$\q\rra \g$ on $\pp=\k^\perp/\k$ is described by the formulas above. To prove the theorem, one has to show that the $\VB$-groupoid action is a Dirac morphism. 

Rather than proving this directly, we use the following argument. Consider the 
action of $(H,\AA,E)$ on $(H,\wh{\AA},\wh{E})$. Recall that this action preserves 
the map $\ol{f}_{\wh{\AA}}\colon \wh{\AA}\to \dd$, and that 
$\ol{f}_{\wh{\AA}}(\gamma_{\wh{\AA}})=-\beta$. Extend to an action on the direct product $(H,\wh{\AA}\times \nn,\wh{E}\times \uu)$, using the trivial action on the 
second factor. The map 
\[ (\ol{f}_{\wh{\AA}},f_\nn)\colon \wh{\AA}\times \nn \to \dd\times \dd\]
is invariant under the action, and takes $(\gamma_{\wh{\AA}},\gamma_\nn)$
to $(-\beta,\beta)$. Since the diagonal $\dd_\Delta\subseteq \dd\times\dd$ 
is a $(-\beta,\beta)$-coisotropic Lie subalgebra, its pre-image 
\[ C=(\ol{f}_{\wh{\AA}},f_\nn)^{-1}(\dd_\Delta)\subseteq \wh{\AA}\times \nn\]
is coisotropic and involutive. The diagonal action of $K$ has generators given by the diagonal embedding of $\k$. We put 
\[ \PP=(C/C^\perp)\qu K,\]
and let 
\[ L=\big(C\cap (\wh{E}\times \uu)/C^\perp\cap(\wh{E}\times \uu)\big)/K.\]
By construction, this Dirac manifold comes equipped with a Dirac action of $(H,\AA,E)$. 

To verify that these two constructions of $(\PP,L)$ coincide, recall that 
$(\wh{\AA},\wh{E})$ is the action Dirac structure 
$(H\times \wh{\q},\,H\times \wh{\g})$, for the action $\varrho\circ f_{\wh{\q}}$. 
Here $f_{\wh{\q}}\colon \wh{\q}=\dd\ltimes \dd^*_\beta\to \wh{\g}=\dd$ is simply the projection
to the first factor. One has $f_{\wh{\AA}}(h,\lambda+\mu)=\lambda$. 
The map 
\[ H\times \nn\to C\subseteq \wh{\AA}\times \nn,\ \ (h,\zeta)\mapsto ((h,f_\nn(\zeta)),\zeta)\] 
preserves metrics, the anchor map, and the brackets, and descends to an isomorphism of metrized vector bundles $H\times \n\to C/C^\perp$, compatible with the anchor maps and brackets. Furthermore, this isomorphism is also compatible with the action of $K$. 

\subsection{Robinson's classification}
We continue with the setting of Theorem \ref{th:class}. Since  $f_\nn(\gamma_\nn)=\beta$, and 
$\uu\subseteq \nn$ is Lagrangian, the image  $\cc:=f_\nn(\uu)\subseteq \dd$  is $\beta$-coisotropic; see \ref{subsubsec:coiso}. According to Proposition \ref{prop:addendum}, the group $\AA/E\cong E^*$ acts transitively on $\PP/L\cong L^*$ if and only if  $f_\nn$ restricts to an isomorphism from $\uu$ onto $\cc$, 
and $\cc\cap \h=\k$.  By the general construction from Section \ref{subsubsec:LApair}, the 
Dirac-Manin pair $(\dd,\cc)_\beta$ determines a Manin pair, given as the reduction of 
$(\dd\ltimes \dd^*_\beta,\dd)$ by the coisotropic Lie subalgebra 
$\cc\ltimes \dd^*_\beta$, with $\cc$ as a Lagrangian Lie subalgebra. 
\begin{lemma}\label{lem:rob}
There is a canonical isomorphism of Manin pairs
\[ (\nn,\uu)\cong \big((\cc\ltimes \dd^*_\beta)/(\cc\ltimes \dd^*_\beta)^\perp,\ \cc).\] 
\end{lemma}
\begin{proof}
We have $(\on{ran}(f_\nn^*)+\uu)^\perp=\ker(f_\nn)\cap \uu^\perp
=\ker(f_\nn)\cap\uu=0$, hence $\on{ran}(f_\nn^*)+\uu=\nn$.  
It follows that any element of $\nn$ can be written in the form 
$x=\zeta+f_\nn^*(\mu)$ with $\zeta\in \uu$ and $\mu\in \dd^*$. 
Note that 
\[ f_\nn^*(\mu)\in \uu\equiv \uu^\perp \Leftrightarrow 
\mu \in \on{ann}(f_\nn(\uu))=\on{ann}(\cc)\Rightarrow f_\nn(f_\nn^*(\mu))=\beta^\sharp(\mu)\in\cc,
\]
Thus, if $\zeta\in \uu,\ \mu\in \dd^*$ with $\zeta+f_\nn^*(\mu)=0$, then 
\[ f_\nn(\zeta)+\mu=-\beta^\sharp(\mu)+\mu\in (\cc\ltimes \dd^*_\beta)^\perp.\]
We hence obtain a well-defined linear map
\[ \nn\to (\cc\ltimes \dd^*_\beta)/ (\cc\ltimes \dd^*_\beta)^\perp, \ \ 
\zeta+f_\nn^*(\mu)\mapsto [f_\nn(\zeta)+\mu].\]
It is straightforward to verify that this map preserves Lie brackets and metrics. For dimension reasons, it is hence an isomorphism. 
\end{proof}
Using Lemma \ref{lem:rob}, 
Theorem \ref{th:class} specializes to the following result, which was first obtained by P. Robinson using a different approach: 
\begin{theorem}[P. Robinson \cite{rob:cla}] 
The Dirac structures $(H/K,\PP,L)$ together with Dirac actions of $(H,\AA,E)$, such that $E^*$ acts transitively on $L^*$, are classified by $K$-invariant $\beta$-coisotropic Lie subalgebras 
$\cc\subseteq\dd$ such that $\cc\cap \h=\k$. 
\end{theorem}
Note that for transitive Poisson actions of Poisson Lie groups $H$ on Poisson manifolds $M$, the condition 
in the theorem is automatic: Here the action of $E^*=TH$ on $L^*=TM$ is just the tangent lift of the 
$H$-action on $M$. 

\subsection{The exact case}
Recall that a  Courant algebroid over a manifold $Q$ is \emph{exact} \cite{sev:let} if the kernel of its anchor map is a Lagrangian subbundle. In this case, the choice of a complementary Lagrangian subbundle identifies the Courant algebroid with the standard Courant algebroid $\T Q$, with the Courant bracket 
twisted by a closed 3-form. 

As we already remarked, \cite[Proposition 7.1]{lib:dir} asserts that the Courant algebroid $\AA$ of a Dirac Lie group $(H,\AA,E)$ is exact if and only if its Dirac-Manin-triple $(\dd,\g,\h)_\beta$  is exact, 
that is, $\beta$ is non-degenerate and $\g$ is Lagrangian. 
As discussed in \cite{lib:dir}, there is in fact a \emph{canonical} splitting $TH\to H\times \dd$, 
given on left-invariant vector fields by 
\[ \nu^L\mapsto (h,\nu-\hh \Ad_{h^{-1}}(1-\pr_{\h^\perp}) \Ad_h \nu).
\]
The corresponding 3-form in this case is expressed in terms of the Maurer-Cartan forms $\theta^L,\theta^R$ as 
\[ \eta=\f{1}{12}\l\theta^L,[\theta^L,\theta^L]\r\]
(where $\l\cdot,\cdot\r$ is the inner product on $\dd$ defined by $\beta$). Thus $\AA\cong \T H_\eta$, in the notation from \cite{lib:dir}. 
\begin{proposition}
Let $(H,\AA,E)$ be an exact Dirac Lie group, classified by $(\dd,\g,\h)_\beta$, and $K\subset H$ a closed subgroup. Then the  \emph{exact} Dirac structures $(\PP,L)$ on $H/K$, together with Dirac actions of $(H,\AA,E)$, are classified by Lagrangian Lie subalgebras $\cc\subseteq \dd$. 
\end{proposition}
\begin{proof}
Let $(\nn,\uu)_{\gamma_\nn}$ and $f_\nn\colon \nn\to \dd$ be the data from the general classification theorem \ref{th:class}, and $\cc=f_\nn(\uu)$. Note that 
$\on{rank}(\PP)=\dim\nn-2\dim\k$. Hence, if $\PP$ is exact as well, we must have 
$\dim\nn=2\dim H=\dim\dd$. Since $f_\nn(\gamma_\nn)=\beta$, this implies that $f_\nn$ 
is an isomorphism from $\nn$ to $\dd$, restricting to an isomorphism from $\uu$ to $\cc$. 
But these conditions are also sufficient: if $\nn=\dd$, then its pullback $p^!\PP$  under the map $p\colon H\to H/K$ equals $H\times \dd=\AA$ as a Courant algebroid. Hence $\PP=(p^!\PP)\qu K$, being the reduction of an exact Courant algebroid, is again exact \cite{bur:red}. 
\end{proof}
We hence see that the Courant algebroid $\PP$ is  $\AA\qu K=\T H_\eta\qu K$, where the reduction is defined using the generators 
\[ \tau\mapsto \tau^L+\hh \l\theta^L,\Ad_{h^{-1}}(1-\pr_{\h^\perp})\Ad_h \tau\r.\]
The reduction is isomorphic to $\T(H/K)$ with the Courant bracket twisted by a closed 3-form; the 
latter  depend on the choice of a splitting. In the special case $\k\subseteq \h^\perp$ the splitting of $\AA$ descends to a splitting of $\PP$, and the 3-form $\eta$  descends to $H/K$.  In general, 
the splitting of $\AA$ does not directly descend, but the choice of a $K$-principal connection on 
$H\to H/K$ determines a reduced splitting for $\PP$ \cite{bur:red}. 

\begin{example}
Suppose $\h$ has an invariant metric, and let $H$ be equipped with the corresponding Cartan-Dirac structure. Recall that the associated $H$-equivariant Dirac Manin triple is 
$(\h\oplus\ol{\h},\h_\Delta,\h\oplus 0)_\beta$. This is an exact Dirac-Manin triple; 
indeed $\AA\cong \T H_\eta$ where $\eta\in \Omega^3(H)$ is the Cartan 3-form determined by the metric. The exact Dirac homogeneous spaces $H/K$ for this Dirac Lie group are given by Lagrangian Lie subalgebra $\mf{c}$ such that $(\h\oplus 0)\cap \mf{c}=\k\oplus 0$. Note that this requires $\k\subset \h$ to be isotropic. If the metric on $\h$ is positive definite, then $\mf{c}$ is obtained as the graph of a Lie algebra automorphism $\kappa$ of $\h$, and one has $\k=0$. 
(Note however that the Dirac structure on $H/K=H$ is a $\kappa$-twisted 
version of the Cartan Dirac structure.) If $H$ is a non-compact real semi-simple Lie group, there are many other examples of Lagrangian Lie subalgebras, with a possibly non-trivial $K$. See Karolinsky-Lyapina  \cite{kar:lag} for some classification results.  Stronger results are available in the complex case, as discussed in Example \ref{subsubsec:wonderful}. 
\end{example}

\begin{appendix}
\section{$\VB$-groupoids}\label{app:vb}
%
%\subsection{}
For any groupoid $H$, we denote by $H^{(0)}$ its space of units, 
and by $\sz,\tz\colon H\to H^{(0)}$ the source and target maps. Sometimes we write 
$\sz_H,\tz_H$ for clarity. The groupoid itself will be written as $H\rra H^{(0)}$, and the groupoid multiplication 
of elements $h_1,h_2$ with $\sz(h_1)=\tz(h_2)$ is written as $h_1\circ h_2$ or simply $h_1h_2$. Throughout, we will take `groupoids' to mean \emph{Lie} groupoids; thus $H$ and $H^{(0)}$ are 
smooth manifolds, all structure maps are smooth, and the source and target maps are surjective submersions.  
We denote by $\on{Gr}(\Mult_H)\subseteq H\times H\times H$ the graph of the groupoid multiplication, consisting of elements 
$(h_1\circ h_2,h_1,h_2)$ with $\sz(h_1)=\tz(h_2)$. 
An \emph{action} of a groupoid $H\rra H^{(0)}$ on a manifold $M$ is given by a smooth map $\uz\colon M\to H^{(0)}$ (called the \emph{moment map}), together with a smooth map $\ca{A}_M\colon (h,m)\mapsto h\circ m$, defined on the submanifold of elements $(h,m)\in H\times M$ such that $\sz(h)=\uz(m)$, such that 
$h_1\circ (h_2\circ m)=(h_1\circ h_2)\circ m$ whenever $\sz(h_1)=\tz(h_2),\ \sz(h_2)=\uz(m)$, and such that  
$h\circ m=m$ whenever $h=\uz(m)\in H^{(0)}$.  We denote by $\on{Gr}(\ca{A}_M)\subseteq M\times H\times M$ the graph of the action map, 
consisting of all $(m',h,m)$ such that $m'=h\circ m$. 
%

%\subsection{}
A \emph{$\VB$-groupoid} \cite{pra:rem} (see also \cite{bur:vec,gra:vb,lib:cou,mac:gen}) is a groupoid $V\rra V^{(0)}$, such that $V$ is a vector bundle (whose base is a groupoid $H\rra H^{(0)}$), 
and $\on{Gr}(\Mult_V)$ is a vector subbundle along $\on{Gr}(\Mult_H)$. The condition on the graph implies that the 
units $V^{(0)}$ are a sub-vector bundle, and that all the groupoid structure maps (multiplication, inversion, and source and target maps) 
are vector bundle morphisms. The \emph{core} of a $\VB$-vector bundle is the vector bundle 
\[ \on{core}(V)=V|_{H^{(0)}}/V^{(0)};\]
the $\VB$-vector bundle is called \emph{vacant} if $\on{core}(V)=0$. The restrictions of $\ker(\sz_V)$ and $\ker(\tz_V)$ to 
$H^{(0)}$ are both complements to $TH^{(0)}$, hence they are both identified with $\on{core}(V)$.

Left translation by elements of $H\subseteq V$ gives a canonical isomorphism $\sz_H^*\on{core}(V)\cong \ker(\tz_V)$; similarly $\tz_H^*\on{core}(V)\cong \ker(\sz_V)$ by right translation.
Composing with the inclusion maps into $V$, and dualizing, we obtain two bundle maps  
\begin{equation}\label{eq:tsv} \sz_{V^*}\colon V^*\to \on{core}(V)^*,\ \ \tz_{V^*}\colon V^*\to \on{core}(V)^*
\end{equation}
with base maps $\sz_H,\ \tz_H$.
These are the source and target maps of a groupoid:
\begin{proposition}[Pradines \cite{pra:rem}] \label{prop:pradines}
For any $\VB$-groupoid $V\rra V^{(0)}$, the dual bundle 
is again a $\VB$-groupoid $V^*\rra (V^*)^{(0)}$, with source and target map 
\eqref{eq:tsv}. The groupoid structure satisfies 
$\l\mu_1\circ \mu_2,v_1\circ v_2\r=\l\mu_1,v_1\r+\l \mu_2,v_2\r$ for elements $\mu_i \in V^*_{h_i}$ and 
$v_i\in V_{h_i}$ such that the compositions are defined. 
\end{proposition}
As an immediate consequence, we see that a $\VB$-groupoid is vacant if and only if the dual $\VB$-groupoid is a group.
\begin{example}\label{ex:part1}
For any Lie groupoid $H\rra H^{(0)}$, the tangent bundle is a $\VB$-groupoid $TH\rra TH^{(0)}$, with core 
the Lie algebroid $AH$ of $H$. Hence, the cotangent bundle of $H$ is a $\VB$-groupoid $T^*H\rra A^*H$, with core the 
cotangent bundle of $H^{(0)}$. It is the symplectic groupoid integrating the Poisson manifold $A^*H$.

If $H^{(0)}=\pt$ so that $H$ is a group, the tangent bundle is a group, and the cotangent bundle is the 
vacant $\VB$-groupoid $T^*H\rra \h^*$. The 
condition $\alpha=\alpha_1\circ\alpha_2$ for $\alpha_i\in T^*_{h_i}H$ and 
$\alpha=T^*_hH$ holds if and only if $h=h_1h_2$ and $(\alpha_1,\alpha_2)=(T_{h_1,h_2}\on{Mult}_H)^*\alpha$. 
We conclude that the graph of the groupoid multiplication $\Mult_{\T H}$ of 
\[ \T H=TH \oplus T^*H \rra \h^*\]
coincides with the graph of the 
Courant morphism $\T \Mult_H\colon \T H\times \T H\da \T H$. 
\end{example}

%\subsection{$\VB$-modules and their duals}\label{subsubsec:dual}
Let $V\rra V^{(0)}$ be a $\VB$-groupoid with base $H\rra H^{(0)}$, and let $P$ be a vector bundle with base $M$, 
with a groupoid action $\ca{A}_P$ of $V$. Then $P$ is called a \emph{$\VB$-module} if the graph 
$\on{Gr}(\ca{A}_P)\subseteq P\times V\times P$
of is a sub-vector bundle.
In particular, the action $\ca{A}_P$  restricts to a groupoid action $\ca{A}_M$ of $H$ on $M$, and the 
moment map 
$\uz_P\colon P\to V^{(0)}$ is a vector bundle morphism along $\uz_M\colon M\to H^{(0)}$.  (Similarly, we define $\LA$-modules over $\LA$-groupoids and $\CA$-modules over $\CA$-groupoids.) The inclusion 
\[ \uz_M^*\on{core}(V)\cong 
\uz_M^*\ker(\sz_V)\to  P,\ \ \ v\mapsto v\circ 0_m,\] dualizes to define a bundle map 
\begin{equation}\label{eq:peq}
\uz_{P^*}\colon P^*\to \on{core}(V)^*=(V^*)^{(0)},
\end{equation}
with base map $\uz_M$. We have the following addendum to Proposition \ref{prop:pradines}: 
\begin{proposition}
For any  $\VB$-module $P$ over $V\rra V^{(0)}$, the dual bundle $P^*$ is a 
$\VB$-module over $V^*\rra (V^*)^{(0)}$, for the moment map \eqref{eq:peq}.
The module action is uniquely defined by the property
\begin{equation}\label{eq:form1}
\l \mu\circ\nu,\ v\circ w\r=\l\mu,v\r+\l \nu,w\r
\end{equation}  
for $\mu\in V^*_{h},\ \nu\in P^*_{m},\ v\in V_{h},\ w\in P_{m}$. 
\end{proposition}
\begin{proof}
(Cf.~ \cite{rob:cla}.) To check that the formula for the module action is well-defined, we must verify that the right hand side vanishes whenever
$v\circ w=0_{h\circ m}$. Since this condition implies $\tz_V(v)=\uz_\PP(0_{h\circ m})=0$, we may write 
$v=0_{h^{-1}}\circ v_1$ with $v_1\in \ker(t_V)_{H^{(0)}}$ %_{\sz_H(h)}$.
%and we have that $v_1\circ w=0_m$. 
Then 
\[ 0_m=v_1\circ w=(v_1-\sz_V(v_1))\circ 0_m+\sz_V(v_1)\circ w=(v_1-\sz_V(v_1))\circ 0_m+ w.\]
Pairing with $\nu$, this shows $\l \nu,(v_1-\sz_{V}(v_1))\circ 0_m\r+\l \nu,w\r=0$. Letting 
$[v_1]\in \on{core}(V)$ be the equivalence class of $v_1$ mod $V^{(0)}$, we have 
\[ \l \nu,(v_1-\sz_{V}(v_1))\circ 0_m\r=\l \uz_{P^*}(\nu),[v_1]\r=\l \sz_{V^*}(\mu),[v_1]\r=
 \l\mu,v\r.\]
This shows that $\l\mu,v\r+\l \nu,w\r=0$ as desired. Hence \eqref{eq:form1} is well-defined; it is straightforward to check that \eqref{eq:form1} gives a $\VB$-groupoid action. 
\end{proof}

\begin{example}
Continuing Example \ref{ex:part1}, suppose the Lie groupoid $H$ acts on a manifold $M$, with a moment map 
$\uz\colon M\to H^{(0)}$. Then $TH\rra TH^{(0)}$ acts on $TM$, and dually $T^*H\rra A^*H$ acts on 
$T^*M$, with a moment map $\uz_{T^*M}\colon T^*M\to A^*H$. 

If $H$ is a Lie group, then this is the usual moment map from symplectic geometry. We obtain an action 
of $\T H\rra \h^*$ on $\T M$, such that the graph of the action map $\mathsf{a}_{\T M}$ coincides with the 
graph of $\T \mathsf{a}_M$.
\end{example}

\section{Semidirect products}\label{app:semidirect}
There is a general notion of \emph{semi-direct product} of two groupoids,  with one groupoid $H$ acting on a second groupoid $G$ by automorphisms (in a suitable sense). We will only need the simple case that $H$ is a Lie group, where this notion is a rather straightforward generalization of the semi-direct product of two Lie groups. 

\begin{proposition}[Semi-direct product]
Suppose $G\rra M$ is a groupoid on which a Lie group $H$ acts by automorphisms. Then there is 
\emph{semi-direct product groupoid} $H\ltimes G\rra M$, equal to $H\times G$ as a manifold, 
with source and target maps 
\[ \sz(h,g)=\sz(g),\ \tz(h.g)=h.\tz(g),\]
and 
with the groupoid multiplication of composable elements given as 
\[ (h_1,g_1)(h_2,g_2)=(h_1h_2, (h_2^{-1}.g_1)g_2).\]
Given a $G$-action on a manifold $Q$, with moment map $\uz_Q\colon Q\to M$, we obtain 
an action of $H\ltimes G$ on $H\times Q$, with moment map $\uz_{H\times Q}(h,q)=h\bullet \uz_Q(q)$, and for composable elements
\[ (h_1,g)\circ (h_2,q)=(h_1h_2,\ (h_2^{-1}.g)\circ q).\]
If a closed subgroup $K\subseteq H$ acts on $Q$, in such a way that $k.(g\circ q)=(k.g)\circ (k.q)$ 
for all $k\in K$ and composable $g\in G,\ q\in Q$, then this action of $H\ltimes G$ descends to  
an action on the associated bundle $H\times_K Q$. 
\end{proposition}
The proof is a straightforward verification. 
%
%\begin{proof}We verify: \[ \sz((h_1,g_1)(h_2,g_2))=\sz((h_2^{-1}.g_1)g_2)=\sz(g_2)=\sz(h_2,g_2),\]\[ \tz((h_1,g_1)(h_2,g_2))=(h_1h_2).\tz((h_2^{-1}.g_1)g_2)=(h_1h_2).h_2^{-1}.g_1=h_1.g_1=\tz(h_1,g_1).\]Furthermore, if $\sz(h_1,g_1)=\tz(h_2,g_2),\ \sz(h_2,g_2)=\tz(h_3,g_3)$, then 
%\[ ((h_1,g_1)(h_2,g_2))(h_3,g_3)=(h_1h_2h_3,\ (h_3^{-1}.(h_2^{-1}.g_1)g_2)g_3=(h_1h_2h_3,\ ((h_2h_3)^{-1}.g_1) (h_2^{-1}.g_2) g_3),\]which is also the result for $(h_1,g_1)((h_2,g_2)(h_3,g_3))=(h_1,g_1)(h_2h_3,(h_3^{-1}.g_2)g_3)$. The rest is clear; in particular $(h,g)^{-1}=(h^{-1},(h.g)^{-1})$.  \end{proof}
As a special case, suppose $H$ acts on a $\VB$-groupoid $V\rra V^{(0)}$ over $G\rra G^{(0)}$. 
Then the semi-direct product $H\ltimes V$ is a $\VB$-groupoid over the semi-direct product 
$H\ltimes G$. 
\begin{example}
The Dirac Lie group $(H,\AA,E)$ determines a linear groupoid 
$\q\rra \g$, with an action of $H$ by automorphisms. Here $\tz(\lambda)=(1-\pr_\rr)(\lambda),\ 
\sz(\lambda)=(1-\pr_{\rr^\perp})(\lambda)$. The action of $H$ preserves the metric, is the $\bullet$-action on $\g$, and preserves both $\rr,\rr^\perp$. The groupoid multiplication of $\q$ reads as 
\[ \lambda_1\circ \lambda_2=\lambda_1+\lambda_2-\sz(\lambda_1).\]
Hence, for the semi-direct product we obtain 
\[ \sz(h,\lambda)=(1-\pr_{\rr^\perp})(\lambda),\ \ 
\tz(h,\lambda)=h\bullet (1-\pr_\rr)(\lambda),\]
and for composable elements, 
\[ (h_1,\lambda_1)\circ (h_2,\lambda_2)=(h_1h_2,\ (h_2^{-1}\lambda_1)\circ \lambda_2)
=(h_1h_2,\lambda_2+h_2^{-1}\bullet (\lambda_1-\sz(\lambda_1))).\]
We conclude that as a groupoid, $\AA$ is a semi-direct product $\AA=H\ltimes \q\rra \g$. 
Likewise $E=H\ltimes \g$. 

Let $\nn$ be a given space, with a moment map $f_\nn\colon \nn\to \dd$. Let 
$\uz=\pr_\g\circ f_\nn$, and let $\q\rra \g$ act by $\lambda\circ \zeta=\zeta+
F_\nn(\lambda)$, where $F_\nn$ was defined in Section \ref{subsec:normal}. It is easily seen that this action is compatible with metrics. 
We obtain an action of $\AA=H\ltimes \q$ on $\PP=H\times \nn$, as explained above. 
\end{example}
\end{appendix}

\bibliographystyle{amsplain}
%\bibliography{../../../Biblio/ref}
%\bibliography{../../Biblio/ref}
\def\cprime{$'$} \def\polhk#1{\setbox0=\hbox{#1}{\ooalign{\hidewidth
  \lower1.5ex\hbox{`}\hidewidth\crcr\unhbox0}}} \def\cprime{$'$}
  \def\cprime{$'$} \def\cprime{$'$} \def\cprime{$'$}
  \def\polhk#1{\setbox0=\hbox{#1}{\ooalign{\hidewidth
  \lower1.5ex\hbox{`}\hidewidth\crcr\unhbox0}}} \def\cprime{$'$}
  \def\cprime{$'$} \def\cprime{$'$} \def\cprime{$'$} \def\cprime{$'$}
\providecommand{\bysame}{\leavevmode\hbox to3em{\hrulefill}\thinspace}
\providecommand{\MR}{\relax\ifhmode\unskip\space\fi MR }
\providecommand{\MRhref}[2]{%
  \href{http://www.ams.org/mathscinet-getitem?mr=#1}{#2}
}
\providecommand{\href}[2]{#2}

\end{document}